\documentclass[12pt]{amsart}
\usepackage{amssymb}
\usepackage{mathabx}
\usepackage{xypic}
\usepackage{hhline}
\usepackage{dcpic}
\usepackage{hyperref}

\newtheorem{theorem}{Theorem}
\newtheorem{lemma}[theorem]{Lemma}

\newtheorem{proposition}[theorem]{Proposition}
\newtheorem{corollary}[theorem]{Corollary}

\newtheorem{remark}[theorem]{Remark}

\numberwithin{theorem}{section}
\numberwithin{equation}{section}

\begin{document}

\makeatletter
\def\Ddots{\mathinner{\mkern1mu\raise\p@
\vbox{\kern7\p@\hbox{.}}\mkern2mu
\raise4\p@\hbox{.}\mkern2mu\raise7\p@\hbox{.}\mkern1mu}}
\makeatother

\newcommand{\OP}[1]{\operatorname{#1}}
\newcommand{\GO}{\OP{GO}}
\newcommand{\leftexp}[2]{{\vphantom{#2}}^{#1}{#2}}
\newcommand{\leftsub}[2]{{\vphantom{#2}}_{#1}{#2}}
\newcommand{\rightexp}[2]{{{#1}}^{#2}}
\newcommand{\rightsub}[2]{{{#1}}_{#2}}
\newcommand{\AI}{\OP{AI}}
\newcommand{\gen}{\OP{gen}}
\newcommand{\prim}{\OP{\star prim}}
\newcommand{\Image}{\OP{Im}}
\newcommand{\Spec}{\OP{Spec}}
\newcommand{\Ad}{\OP{Ad}}
\newcommand{\tr}{\OP{tr}}
\newcommand{\spec}{\OP{spec}}
\newcommand{\scopy}{\OP{end}}
\newcommand{\ord}{\OP{ord}}
\newcommand{\Cent}{\OP{Cent}}
\newcommand{\wellip}{\OP{w-ell}}
\newcommand{\Nrd}{\OP{Nrd}}
\newcommand{\Res}{\OP{Res}}
\newcommand{\alg}{\OP{alg}}
\newcommand{\BC}{\OP{BC}}
\newcommand{\sgn}{\OP{sgn}}
\newcommand{\SU}{\OP{SU}}
\newcommand{\Hom}{\OP{Hom}}
\newcommand{\Inter}{\OP{Int}}
\newcommand{\diag}{\OP{diag}}
\newcommand{\Sym}{\OP{Sym}}
\newcommand{\GSp}{\OP{GSp}}
\newcommand{\GL}{\OP{GL}}
\newcommand{\GSO}{\OP{GSO}}
\newcommand{\height}{\OP{ht}}
\newcommand{\art}{\OP{art}}
\newcommand{\vol}{\OP{vol}}
\newcommand{\cusp}{\OP{cusp,\tau}}
\newcommand{\un}{\OP{un}}
\newcommand{\disci}{\OP{disc,\tau_{\it i}}}
\newcommand{\cuspi}{\OP{cusp,\tau_{\it i}}}
\newcommand{\ellip}{\OP{ell}}
\newcommand{\sph}{\OP{sph}}
\newcommand{\gsimp}{\OP{sim-gen}}
\newcommand{\Aut}{\OP{Aut}}
\newcommand{\disc}{\OP{disc,\tau}}
\newcommand{\sdisc}{\OP{s-disc}}
\newcommand{\aut}{\OP{aut}}
\newcommand{\End}{\OP{End}}
\newcommand{\barQ}{\OP{\overline{\mathbf{Q}}}}
\newcommand{\barQp}{\OP{\overline{\mathbf{Q}}_{\it p}}}
\newcommand{\Gal}{\OP{Gal}}
\newcommand{\PGL}{\OP{PGL}}
\newcommand{\simp}{\OP{sim}}
\newcommand{\pri}{\OP{prim}}
\newcommand{\Normal}{\OP{Norm}}
\newcommand{\Ind}{\OP{Ind}}
\newcommand{\St}{\OP{St}}
\newcommand{\unit}{\OP{unit}}
\newcommand{\reg}{\OP{reg}}
\newcommand{\SL}{\OP{SL}}
\newcommand{\Frob}{\OP{Frob}}
\newcommand{\Id}{\OP{Id}}
\newcommand{\GSpin}{\OP{GSpin}}
\newcommand{\Norm}{\OP{Norm}}

\title[On a theorem of Bertolini-Darmon]{On a theorem of Bertolini-Darmon about rationality of Stark-Heegner points over genus fields of real quadratic fields}
\author{Chung Pang Mok}

\address{Department of Mathematics, Purdue University, 150 N. University Street, West Lafayette, IN 47907-2067}

\email{mokc@purdue.edu}

\maketitle

\begin{abstract}
In this paper, we remove certain hypothesis in the theorem of Bertolini-Darmon on the rationality of Stark-Heegner points over narrow genus class fields of real quadratic fields. Along the way, we establish that certain normalized special values of $L$-functions are squares of rational numbers, a result that is of independent interest, and can be regarded as instances of the rank zero case of the Birch and Swinnerton-dyer conjecture modulo squares.
\end{abstract}

\section{Introduction}

Let $E/\mathbf{Q}$ be an elliptic curve over $\mathbf{Q}$ whose conductor we denote as $N$, and $f_E$ being the weight two cuspidal eigenform of level $N$ that is attached to $E/\mathbf{Q}$. We assume that $E$ has multiplicative reduction at an odd prime $p \| N $, which will be fixed in this Introduction section. Put $M:=N/p$. Let $\chi_1,\chi_2$ be a pair of quadratic (or trivial) Dirichlet characters, corresponding to quadratic (or trivial) field extensions of $\mathbf{Q}$ with discriminants $D_1,D_2$ respectively. Assume that $D_1,D_2$ are relatively prime, and that $D_1,D_2$ are relatively prime to $N$. Put $\epsilon := \chi_1 \cdot \chi_2$, which is assumed to be non-trivial. Denote by $K$ the quadratic field extension of $\mathbf{Q}$ that corresponds to the Dirichlet character $\epsilon$, which has discriminant equal to $D:=D_1 \cdot D_2$. The pair $(\chi_1,\chi_2)$ defines a narrow genus class character $\chi$ of $K$, and corresponds to the quadratic (or trivial) extension $H_{\chi} = \mathbf{Q}(\sqrt{D_1},\sqrt{D_2})$ of $K$ that is unramified at all the finite primes of $K$, namely the narrow genus class field of $K$ corresponding to the pair $(\chi_1,\chi_2)$. We also regard the narrow genus class character $\chi$ as a character of $\Gal(H_{\chi}/K)$. 

\bigskip

Now in addition, we assume that $K$ is real quadratic, and satisfies the modified Heegner hypothesis:  
\begin{itemize}
\item The prime $p$ is inert in $K$.

\item All primes dividing $M$ splits in $K$. 
\end{itemize}

Denote by $K_p$ the completion of $K$ at the prime ideal $p \mathcal{O}_K$. Note that the prime ideal $p \mathcal{O}_K$ splits in $H_{\chi}$. By fixing a prime of $H_{\chi}$ that lies over $p\mathcal{O}_K$, we can embed $H_{\chi}$ into the completion $K_p$.

\bigskip

As a special case of the theory in [D], one can give a $p$-adic analytic construction of certain points $P_{\chi} \in E(K_p)$ associated to the data $f_E$ and $\chi$. The point $P_{\chi}$ is a certain linear combination (weighted by the values taken by $\chi$) of Stark-Heegner points that were defined in {\it loc. cit.} It is conjectured that, up to non-zero integer multiples, the point $P_{\chi}$ lies in $E(H_{\chi})$, and also that the $P_{\chi}$ is non-torsion if and only if $L^{\prime}(1,E/K,\chi) \neq 0$; here $L(s,E/K,\chi)$ is the complex Rankin-Selberg $L$-function attached to the pair $f_E$ and $\chi$. 

\bigskip

In [BD1], this was established under certain hypotheses. To state their results, we first need some more setup.

\bigskip
The condition that $K$ is real quadratic, together with the modified Heegner hypothesis, implies that $\epsilon(-M)=1$, and hence that $\chi_1(-M)=\chi_2(-M)$. Finally, denote by $w_M$ the sign of the Atkin-Lehner involution at $M$ acting on $f_E$.  

\bigskip

The main theorem of [BD1] could be stated as follows:

\begin{theorem} (Theorem 4.3 of [BD1])
Assume in addition the following conditions hold:
\begin{itemize}
\item $E/\mathbf{Q}$ has multiplicative reduction at some prime other than $p$. 
\item $\chi_1(-M)=-w_M$.
\end{itemize}
Then there is a global point $\mathbf{P}_{\chi} \in (E(H_{\chi}) \otimes \mathbf{Q})^{\chi}$, and $t \in \mathbf{Q}^{\times}$, such that $P_{\chi} = t \cdot \mathbf{P}_{\chi}$ as elements in $E(K_p) \otimes \mathbf{Q}$. Furthermore, $P_{\chi}$ is non-torsion if and only if $L^{\prime}(1,E/K,\chi) \neq 0$.
\end{theorem}

\bigskip

The main theorem of this paper is as follows:
\begin{theorem}
Theorem 1.1 holds, without having to assume that $E$ has multiplicative reduction at some prime other than $p$. 
\end{theorem}

Here, the main point being that, for the proof of Theorem 4.3 in [BD1], a crucial input is the main result in [BD2] (to be recalled in section 2 below), which was established under the condition that $E/\mathbf{Q}$ has multiplicative reduction at some prime other than $p$, and this is also the {\it only place} where [BD1] had to assume this condition on $E/\mathbf{Q}$. In this paper, we show that the main result of [BD2] holds without having to assume that $E/\mathbf{Q}$ has multiplicative reduction at some prime other than $p$. Once this is done, then the same argument used for the proof of Theorem 4.3 of [BD1] go through, without having to assume that $E/\mathbf{Q}$ has multiplicative reduction at some prime other than $p$. On the other hand, the condition $\chi_1(-M) = -w_M$ seems to be quite intrinsic to the method used in [BD1] and we are not able to remove it at the present moment. 

\bigskip

The organization of this paper is as follows. In section 2, we recall the main result of [BD2], and also our previous work [M1], where we were able to establish the main result of [BD2] without having to assume that $E/\mathbf{Q}$ has multiplicative reduction at some prime other than $p$, {\it provided that} we could show that certain normalized special values of $L$-functions that were studied in [M2], is a square of a rational number, {\it c.f.} Theorem 2.2 and Corollary 2.3. The remaining parts of the paper are then independent of sections 1 and 2. In section 3, we study this type of normalized special values of $L$-functions in a more general setting, which is of interest on its own, especially in view of the fact that, the Birch and Swinnerton-dyer conjecture also implies that these normalized special $L$-values are squares of rational numbers, up to factor of $2$. Modulo certain hypothesis which will be harmless for establishing Theorem 2, we show that these normalized special $L$-values are indeed squares of rational numbers, {\it c.f.} Theorem 3.2 for the precise statement. This is achieved by combining the arguments that are generalization of those in [M2], with the results in section 4, concerning the extensions of some of the results of Ribet-Takahashi [RT] and Takahashi [Tak] to the setting of Shimura curves associated to quaternion algebras over totally real fields [De].

\section*{Notations}

For a number field $F$, we will denote its ring of integers as $\mathcal{O}_F$, and a place of $F$ will in general be noted as $v$. The finite places, which correspond to the prime ideals of $\mathcal{O}_F$, will also be denoted as $\mathfrak{l},\mathfrak{p},\mathfrak{q}$ etc. The completion of $F$ at $v$ is noted as $F_v$, and $\pi_v$ denotes a local uniformizer of $F_v$.

\bigskip

If $V$ is a vector space over a field $L$ equipped with an action of a group $G$, then for $\chi$ a character of $G$ with values in $L$, we denote by $V^{\chi}$ the $\chi$-eigenspace of $V$ with respect to the action of $G$. 

\section{R\'esum\'e on the results in [BD2] and [M1]}

In this section, we recall the main results from [BD2] and [M1]. As before $E/\mathbf{Q}$ is an elliptic curve with conductor $N$, that $E/\mathbf{Q}$ has multiplicative reduction at a fixed odd prime $p \|N$, and $f_E$ is the weight two cuspidal eigenform of level $N$ attached to $E/\mathbf{Q}$. Denote $M:=N/p$. The sign $w_N$ of the Atkin-Lehner involution at $N$ acting on $f_E$ has a decomposition $w_N =w_M \cdot w_p$, where $w_M$ and $w_p$ are signs of the Atkin-Lehner involution at $M$ and respectively at $p$ acting on $f_E$. The sign of the functional equation for the complex $L$-function $L(s,E/\mathbf{Q})=L(s,f_E)$ is equal to $-w_N$. 

\bigskip
Let $a_p = a_p(E)$ be the $U_p$-eigenvalue of $f_E$. Then since $E/\mathbf{Q}$ is multiplicative at $p$, one has $a_p = -w_p = \pm1$, and thus in particular $f_E$ is ordinary at $p$. Let $f_{\infty}=\{f_k\}$ be the $p$-adic Hida family lifting $f_E$; so for integer $k \geq 2$ that is sufficiently close to $2$ in the weight space, $f_k$ is a $p$-ordinary cuspidal eigenform of weight $k$, and such that $f_2 =f_E$. Let also $\chi_1$ be a Dirichlet character whose conductor is relatively prime to $N$. One has the Mazur-Tate-Teitelbaum [MTT] $p$-adic $L$-functions $L_p(s,f_k,\chi_1)$ attached to $f_k$ and $\chi_1$, with respect to choices of complex period factors $\Omega^{\pm}_{f_k}$ (which does not depend on $\chi_1$) for the newform associated to $f_k$. Furthermore, following Greenberg-Stevens [GS], the complex periods $\Omega_{f_k}^{\pm}$ can be chosen in such a way that the $p$-adic $L$-functions $L_p(s,f_k,\chi_1)$ associated to the members of the Hida family $f_{\infty} =\{ f_k \}$ can be packaged into a two variable $p$-adic $L$-function, i.e. there exists a $p$-adic analytic function $L_p(s,k,f_{\infty},\chi_1)$ of the two $p$-adic variables $s$ and $k$, such that for integer $k \geq 2$ sufficiently close to $2$ in the weight space, one has $L_p(s,k,f_{\infty},\chi_1)= L_p(s,f_k,\chi_1)$. 

\bigskip

Suppose now that the Dirichlet character $\chi_1$ is quadratic (or trivial), which thus corresponds to a quadratic (or trivial) extension of $\mathbf{Q}$; denote the discriminant of this extension as $D_1$ (thus the conductor of $\chi_1$ is equal to the absolute value of $D_1$). The sign of the functional equation for the complex $L$-function $L(s,E/\mathbf{Q},\chi_1) =L(s,f_E,\chi_1)$ is equal to $-\chi_1(-N) w_N$. We now recall the following theorem from [BD2], that concerns the specialization of the two variable $p$-adic $L$-function $L_p(s,k,f_{\infty},\chi_1)$ to the line $s=k/2$, under the conditions that $\chi_1(p)=a_p$, and also that $\chi_1(-N) = w_N$, i.e. the sign of the functional equation for $L(s,E/\mathbf{Q},\chi_1) =L(s,f_E,\chi_1)$ is equal to $-1$.

\begin{theorem} (Theorem 5.4 of [BD2]) Assume $\chi_1$ satisfies:
\[
\chi_1(-N) = w_N, \mbox{ and } \chi_1(p)=a_p.
\]
In addition, assume that $E/\mathbf{Q}$ has multiplicative reduction at some prime other than $p$. Then we have:
\begin{itemize}
\item The function $L_p(k/2,k,f_{\infty},\chi_1)$ vanishes to order at least $2$ at $k=2$. 
\item There exists a global point $\mathbf{P}_{\chi_1} \in (E(\mathbf{Q}(\sqrt{D_1})) \otimes \mathbf{Q})^{\chi_1}$, and $\ell \in \mathbf{Q}^{\times}$, such that
 \[ \frac{d^2}{dk^2} L_p(k/2,k,f_{\infty},\chi_1) \Big|_{k=2} = \ell (\log_E \mathbf{P}_{\chi_1})^2 \]
\noindent  here $\log_E$ is the $p$-adic formal group logarithm on $E$ (defined as in equation (7) of [BD2]). 
 \item The point $\mathbf{P}_{\chi_1}$ is non-torsion if and only if $L^{\prime}(1,E/\mathbf{Q},\chi_1) \neq 0$. 
 \item The image of $\ell$ in $\mathbf{Q}^{\times}/(\mathbf{Q}^{\times})^2$ is equal that of $L^{\alg}(1,E/\mathbf{Q},\chi_2)=L^{\alg}(1,f_E,\chi_2)$, the algebraic part of the $L$-value $L(1,E/\mathbf{Q},\chi_2) =L(1,f_E,\chi_2)$, where $\chi_2$ is any quadratic Dirichlet character whose conductor is relatively prime to $N$ and also relatively prime to the conductor of $\chi_1$, and satisfying the following conditions:
 
 (a) $\chi_2(-1) = \chi_1(-1)$. 
 
 (b) $\chi_2(l) =\chi_1(l)$ for all primes $l$ dividing $M$.
 
 (c) $\chi_2(p)=-\chi_1(p)$.
 
 (d) $L(1,E/\mathbf{Q},\chi_2) \neq 0$.
 \end{itemize}
\end{theorem}

\bigskip
Here in the final portion of the above theorem, we note, under the condition $\chi_1(-N) =w_N$, one has that for any quadratic Dirichlet character $\chi_2$ satisfying conditions (a), (b), (c) above, that the sign of the functional equation for $L(s,E/\mathbf{Q},\chi_2)=L(s,f_E,\chi_2)$ is equal to $-\chi_2(-N) w_N = \chi_1(-N) w_N=+1$, and thus there exists infinitely many $\chi_2$ satisfying the conditions (a), (b), (c) and (d), by the main result of Murty-Murty [MM] or by Friedberg-Hoffstein [FH]. Secondly, the quantity $L^{\alg}(1,E/\mathbf{Q},\chi_2)=L^{\alg}(1,f_E,\chi_2)$, namely the algebraic part of the $L$-value $L(1,E/\mathbf{Q},\chi_2) =L(1,f_E,\chi_2)$, is defined as:
\[
L^{\alg}(1,E/\mathbf{Q},\chi_2)= L^{\alg}(1,f_E,\chi_2) : = \frac{ c_{\chi_2} L(1,f_E,\chi_2)}{\tau(\chi_2) \Omega^{\chi_2(-1)}_{f_E}}
\]
where $c_{\chi_2}$ is the conductor of $\chi_2$ and $\tau(\chi_2)$ is the Gauss sum of $\chi_2$, while $\Omega_{f_E}^{\pm}$ are the complex periods for $f_E$ that are used for the definition of the $p$-adic $L$-function associated to $f_E$. 
\bigskip

The proof of Theorem 4.3 of [BD1] which is recalled as Theorem 1.1 in the Introduction, relies on the above theorem of [BD2] as one of the main inputs; in particular, the hypothesis that requires $E/\mathbf{Q}$ to have multiplicative reduction at some prime other than $p$ in [BD1], is due to the use of this theorem from [BD2], which requires this condition to be satisfied.

\bigskip

In [M1], we had extended the above main theorem of [BD2] to the setting of elliptic curves over totally real fields $F$ such that $p$ is inert in $F$. By using a base change argument to a suitable real quadratic field $F$, we then obtained in section 6 of [M1] the following result concerning the original situation of $E/\mathbf{Q}$, {\it without} having to assume that $E/\mathbf{Q}$ has multiplicative reduction at some prime other than $p$ (the idea of considering base change to real quadratic fields was already suggested in [BD2]). We state the result that we need from [M1] as follows:

\begin{theorem} ({\it c.f.} p. 929 - 931 of [M1])
As before $E/\mathbf{Q}$ is an elliptic curve with multiplicative reduction at an odd prime $p$ with conductor $N$. Assume again that $\chi_1$ satisfies:
\[
\chi_1(-N) = w_N, \mbox{ and } \chi_1(p)=a_p.
\]
Then we have:
\begin{itemize}
\item The function $L_p(k/2,k,f_{\infty},\chi_1)$ vanishes to order at least $2$ at $k=2$. 
\item There exists a global point $\mathbf{P}^{\prime}_{\chi_1} \in (E(\mathbf{Q}(\sqrt{D_1})) \otimes \mathbf{Q})^{\chi_1}$, and $\ell^{\prime \prime} \in \mathbf{Q}^{\times}$, such that
 \[ \frac{d^2}{dk^2} L_p(k/2,k,f_{\infty},\chi_1) \Big|_{k=2} = \ell^{\prime \prime} (\log_E \mathbf{P}^{\prime}_{\chi_1})^2. \]
\noindent  Here the image of $\ell^{\prime \prime}$ in $\mathbf{Q}^{\times}/(\mathbf{Q}^{\times})^2$ is specified more precisely below. 
\item The point $\mathbf{P}_{\chi_1}^{\prime}$ is non-torsion if and only if $L^{\prime}(1,E/\mathbf{Q},\chi_1) \neq 0$. 
 \end{itemize}
\end{theorem}

\bigskip

Here, according to Remark 6.6 of [M1], the specification of $\ell^{\prime \prime}$ in $\mathbf{Q}^{\times}/(\mathbf{Q}^{\times})^2$ is given as follows. We need some preparations. Firstly, in the context of Theorem 2.2, we also have the condition $\chi_1(-N)=w_N$ as in Theorem 2.1, and thus as before, there exists infinitely many quadratic Dirichlet characters $\chi_2$ whose conductor is relatively prime to $N$, and also relatively prime to the conductor of $\chi_1$, and satisfies the conditions (a), (b), (c), (d) of the final portion of Theorem 2.1; recall that these are:

 (a) $\chi_2(-1) = \chi_1(-1)$. 
 
 (b) $\chi_2(l) =\chi_1(l)$ for all primes $l$ dividing $M :=N/p$.
 
 (c) $\chi_2(p)=-\chi_1(p)$.
 
 (d) $L(1,E/\mathbf{Q},\chi_2) \neq 0$.

\bigskip
Take any such choice of $\chi_2$. The product $\chi_1 \cdot \chi_2$ is then a quadratic Dirichlet character that corresponds to a real quadratic extension $F$ of $\mathbf{Q}$, with the property that $p$ is inert in $F$, while all primes dividing $M$ splits in $F$. Denote by $\mathfrak{p} := p\mathcal{O}_F$ the prime of $\mathcal{O}_F$ above $p$. The pair $(\chi_1,\chi_2)$ also defines a narrow genus class character of $F$, which will be noted as $\chi_F$. When regarded as a Hecke character $\chi_F = \otimes^{\prime}_v \chi_{F,v}$, it is unramified at all the finite primes of $\mathcal{O}_F$, and one has the following: $\chi_{F,v}(-1) = \chi_1(-1) = \chi_2(-1)$ for archimedean $v$; if $q$ is a prime number that is inert in $F$ and $\mathfrak{q} = q\mathcal{O}_F$ the prime of $\mathcal{O}_F$ above $q$, then we have $\chi_{F,\mathfrak{q}} \equiv 1$; in particular we have $\chi_{F,\mathfrak{p}} \equiv  1$. On the other hand, if $l$ is a prime number that splits in $F$, then $\chi_{F,\mathfrak{l}}(\pi_{\mathfrak{l}}) =\chi_1(l)=\chi_2(l)$ for prime $\mathfrak{l}$ of $F$ that lies over $l$. 

\bigskip
We consider the base change of $E$ to $F$. Then $E/F$ is again modular by the theory of base change; more precisely $E/F$ is associated to $\mathbf{f}_E$, the cuspidal Hilbert eigenform over $F$ of parallel weight two with level $N \mathcal{O}_F$, given by the base change of $f_E$ from $\GL_{2/\mathbf{Q}}$ to $\GL_{2/F}$. Then for any Hecke character $\delta$ of $F$, the complex $L$-function $L(s,E/F,\delta)$ is equal to $L(s,\mathbf{f}_E,\delta)$. We have in particular that:
\begin{eqnarray*}
& & L(s,E/F,\chi_F) = L(s,\mathbf{f}_E,\chi_F) \\
&=& L(s,E/\mathbf{Q},\chi_1) \cdot L(s,E/\mathbf{Q},\chi_2) = L(s,f_E,\chi_1) \cdot L(s,f_E,\chi_2). 
\end{eqnarray*}
Now the signs of the functional equations for $L(s,E/\mathbf{Q},\chi_1)$ and $L(s,E/\mathbf{Q},\chi_2)$ are opposite to each other, and hence the sign of the functional equation for $L(s,E/F,\chi_F)$ is equal to $-1$. 

\bigskip

We now consider quadratic Hecke characters $\delta = \otimes^{\prime}_v \delta_v$ of $F$, unramified at the primes of $\mathcal{O}_F$ dividing $N$, and satisfying the conditions:

(1) $\delta_v(-1) = \chi_{F,v}(-1)$ for archimedean $v$. 

(2) $\delta_{\mathfrak{p}}(\pi_{\mathfrak{p}}) = - \chi_{F,\mathfrak{p}}(\pi_{\mathfrak{p}})=-1$

(3) $\delta_{\mathfrak{l}}(\pi_{\mathfrak{l}}) = \chi_{F,\mathfrak{l}}(\pi_{\mathfrak{l}})$ for any prime $\mathfrak{l}$ of $F$ dividing $M=N/p$. 

(4) $L(1,E/F,\delta) \neq 0$. 

\bigskip

Here for any quadratic Hecke character $\delta$ of $F$ satisfying conditions (1), (2), (3), the sign of the functional equation for $L(s,E/F,\delta)=L(s,\mathbf{f}_E,\delta)$ is opposite to that of $L(s,E/F,\chi_F)$, and hence is equal to $+1$, and thus the existence of such a $\delta$ that satisfies (1), (2), (3), (4) follows from the main result of [FH].

\bigskip

We can now finally come back to the portion of Theorem 2.2 concerning the specification of the image of the rational number $\ell^{\prime \prime}$ in $\mathbf{Q}^{\times}/(\mathbf{Q}^{\times})^2$. With the notations as already setup above, we have, according to Remark 6.6 of [M1], that the image of $\ell^{\prime \prime}$ in $\mathbf{Q}^{\times}/(\mathbf{Q}^{\times})^2$ is given by:
\begin{eqnarray}
& & \ell^{\prime \prime} \\ &=& L^{\alg}(1,E/\mathbf{Q},\chi_2) \cdot 2  \frac{D_F^{1/2} (\mathcal{N}_{F/\mathbf{Q}} \mathfrak{c}_{\delta}) L(1,E/F,\delta)}{\tau(\delta) (\Omega_{f_E}^{w})^2} \bmod{(\mathbf{Q}^{\times})^2} \nonumber
\end{eqnarray}

\bigskip

Here $\mathfrak{c}_{\delta}$ is the conductor of $\delta$ and $\mathcal{N}_{F/\mathbf{Q}} \mathfrak{c}_{\delta}$ is its norm, while $\tau(\delta)$ is the Gauss sum of $\delta$. The term $D_F$ is the discriminant of $F$, and finally $w : = \chi_1(-1)=\chi_2(-1)$. 

\bigskip

Thus in order to compare the derivative formula of Theorem 2.2 with that of Theorem 2.1, we need to analyze the quantity:
\begin{eqnarray}
 2  \frac{D_F^{1/2} \mathcal{N}_{F/\mathbf{Q}} \mathfrak{c}_{\delta} L(1,E/F,\delta)}{\tau(\delta) (\Omega_{f_E}^{w})^2}
\end{eqnarray}
The general method of Shimura allows us to prove that (2.2) is a rational number, {\it c.f.} Proposition 3.1 below, but to prove that it is actually the square of a rational number, new technique is required. We will analyze this in section 3 and 4 below (which are independent of section 1 and 2). In Theorem 3.2 below, whose proof will be completed in section 4, we show that (2.2) is indeed the square of a (non-zero) rational number, {\bf under} the conditions that  $L^{\prime}(1,E/\mathbf{Q},\chi_1) \neq 0$ and $L(1,E/\mathbf{Q},\chi_2) \neq 0$ (the condition $L(1,E/\mathbf{Q},\chi_2) \neq 0$ is of course already enforced as condition (d) on $\chi_2$ above). Admitting this for the moment, we thus obtain the:

\begin{corollary}
The statement of Theorem 2.1 holds without having to assume that $E/\mathbf{Q}$ is multiplicative at some prime other than $p$.
\end{corollary}
\begin{proof}
We place ourselves in the situation of Theorem 2.2. Firstly we note, subject to Theorem 3.2 that we will establish below, that when $L^{\prime}(1,E/\mathbf{Q},\chi_1) \neq 0$, then the quantity (2.2) is the square of a (non-zero) rational number; hence by (2.1), the image of $\ell^{\prime \prime}$ in $\mathbf{Q}^{\times}/(\mathbf{Q}^{\times})^2$ is the same as that of $\ell$ as specified in the final portion of Theorem 2.1. Thus the statement of Theorem 2.1 holds. 

\bigskip 
On the other hand, if $L^{\prime}(1,E/\mathbf{Q},\chi_1)=0$, then in the context of Theorem 2.2, the point $\mathbf{P}_{\chi_1}^{\prime}$ is torsion, and so $\log_E \mathbf{P}_{\chi_1}^{\prime}=0$. The derivative formula then shows that $\frac{d^2}{dk^2} L_p(k/2,k,f_{\infty},\chi_1) \big|_{k=2} = 0$. So again the statement of Theorem 2.1 holds; we may simply take $\mathbf{P}_{\chi}$ to be the zero element in $(E(\mathbf{Q}(\sqrt{D_1}) ) \otimes \mathbf{Q})^{\chi_1} $.
\end{proof}

\bigskip

Now with Corollary 2.3 being established, then as discussed in the Introduction, the same arguments in [BD1] used for establishing Theorem 1.1 (i.e. Theorem 4.3 of [BD1]) then carry through without having to assume that $E/\mathbf{Q}$ has multiplicative reduction at some prime other than $p$. Thus we obtain Theorem 1.2, subject to establishing Theorem 3.2 in the next two sections.

\section{Certain special $L$-values}

In this section, we look at certain special $L$-values that in particular include those stated in equation (2.2). The discussions in the present section and the next are independent of the previous sections.

\bigskip
Thus again $E/\mathbf{Q}$ is an elliptic curve with conductor $N$, and $f_E$ being the weight two cuspidal eigenform of level $N$ attached to $E/\mathbf{Q}$, and $\Omega_{f_E}^{\pm}$ be (positive) real and imaginary periods associated to $f_E$. We consider a setup slightly more general than the previous sections, in the following sense. We consider the setting where $N$ can be factored as $N=M \mathcal{Q}$, where $M,\mathcal{Q}$ are relatively prime, and $\mathcal{Q}$ is square-free (the setting in the previous sections is the special case where $\mathcal{Q}=p$). We fix a pair of quadratic (or trivial) Dirichlet characters $\chi_1,\chi_2$, whose conductors are relatively prime to each other and also relatively prime to $N$. The Dirichlet characters $\chi_1,\chi_2$ corresponds to quadratic (or trivial) extension of $\mathbf{Q}$, whose discriminants would be denoted as $D_1,D_2$ respectively. 

\bigskip 

We assume that the following conditions are satisfied (these conditions are the modified Heegner hypothesis in our current setting):
\begin{itemize}
\item $\chi_1(-1)=\chi_2(-1)$. We denote this common sign as $w$.
\item $\chi_1(l)= \chi_2(l)$ for all primes $l$ dividing $M$. 
\item $\chi_1(q)=-\chi_2(q)$ for all primes dividing $\mathcal{Q}$. 
\end{itemize}
We also assume that $\chi_1,\chi_2$ are not {\bf both} trivial; this is automatic from the third item of the above conditions when $\mathcal{Q} \neq 1$. The product $\chi_1 \cdot \chi_2$ is then a quadratic Dirichlet character that corresponds to a real quadratic field $F$, such that all primes dividing $M$ split in $F$, while all primes dividing $\mathcal{Q}$ are inert in $F$. The discriminant of $F$ is denoted as $D_F$. The pair $(\chi_1,\chi_2)$ defines a quadratic (or trivial) narrow genus class character $\chi_F$ of $F$ as before; denote by $H_{\chi_F}$ the quadratic (or trivial) narrow genus class field extension of $F$ that corresponds to $\chi_F$. Recall that $\chi_{F,v}(-1) = \chi_1(-1) = \chi_2(-1) =w$ for archimedean $v$; if $q$ is a prime number that is inert in $F$ and $\mathfrak{q} = q\mathcal{O}_F$ the prime of $\mathcal{O}_F$ above $q$, then we have $\chi_{F,\mathfrak{q}} \equiv 1$. On the other hand, if $l$ is a prime number that splits in $F$, then $\chi_{F,\mathfrak{l}}(\pi_{\mathfrak{l}}) =\chi_1(l)=\chi_2(l)$ for prime $\mathfrak{l}$ of $F$ that lies over $l$. 

\bigskip
Again denote by $\mathbf{f}_E$ the base change of $f_E$ from $\GL_{2/\mathbf{Q}}$ to $\GL_{2/F}$, which is a cuspidal Hilbert eigenform of parallel weight two and conductor equal to $N\mathcal{O}_F$ that corresponds to $E/F$. As before one has:
\begin{eqnarray*}
& & L(s,E/F,\chi_F) = L(s,\mathbf{f}_E,\chi_F) \\
&=& L(s,E/\mathbf{Q},\chi_1) \cdot L(s,E/\mathbf{Q},\chi_2) = L(s,f_E,\chi_1) \cdot L(s,f_E,\chi_2).
\end{eqnarray*}

\bigskip

In the case where $\mathcal{Q}$ is a product of an odd number of primes (which in particular is the case in the setting of previous two sections), then the sign of the functional equations for $ L(s,E/\mathbf{Q},\chi_1)$ and $L(s,E/\mathbf{Q},\chi_2)$ are opposite to each other, and hence the sign of the functional equation for $L(s,E/F,\chi_F)$ is $-1$. On the other hand, in the case, where $\mathcal{Q}$ is a product of an even number of primes,  then the sign of the functional equations for $ L(s,E/\mathbf{Q},\chi_1)$ and $L(s,E/\mathbf{Q},\chi_2)$ are equal, and hence the sign of the functional equation for $L(s,E/F,\chi_F)$ is $+1$. 

\bigskip

We now consider $L(s,E/F,\delta)$, where $\delta = \otimes^{\prime}_v \delta_v$ is a quadratic Hecke character of $F$, satisfying $\delta_v(-1) =w$ for the archimedean places $v$ of $F$. Then firstly, we have:

\begin{proposition} For any quadratic character $\delta  = \otimes^{\prime}_v \delta_v$ of $F$, satisfying $\delta_v(-1) =w$ for the archimedean places $v$ of $F$, we have that the expression:
\begin{eqnarray*} 
 \frac{D_F^{1/2} \mathcal{N}_{F/\mathbf{Q}} \mathfrak{c}_{\delta} L(1,E/F,\delta) }{ \tau(\delta) (\Omega_{f_E}^w)^2}
\end{eqnarray*}
is a rational number.
\end{proposition}
\begin{proof}
As in section 5.1 of [M1], there is a (non-zero) period factor $\Omega_{\mathbf{f}_E}^{w,w}$ associated to the cuspidal Hilbert eigenform $\mathbf{f}_E$, such that for any finite order Hecke character $\delta = \otimes^{\prime}_v \delta_v$ of $F$, satisfying $\delta_v(-1) =w$ for all the archimedean places $v$ of $F$, the expression:
\begin{eqnarray*} 
 \frac{ \mathcal{N}_{F/\mathbf{Q}} \mathfrak{c}_{\delta} L(1,E/F,\delta) }{ \tau(\delta^{-1}) \Omega^{w,w}_{\mathbf{f}_E}}
\end{eqnarray*}
belongs to the field extension $\mathbf{Q}(\mathbf{f}_E,\delta)$ of $\mathbf{Q}$ generated by the Hecke eigenvalues of $\mathbf{f}_E$ and the values taken by $\delta$. But the Hecke eigenvalues of $\mathbf{f}_E$ are rational integers (since $\mathbf{f}_E$ corresponds to the elliptic curve $E/F$), and thus $\mathbf{Q}(\mathbf{f}_E,\delta)$ reduces to the field extension $\mathbf{Q}(\delta)$ of $\mathbf{Q}$ generated by the values of $\delta$. Next, by Lemma 6.1 of [M1], since $\mathbf{f}_E$ arises as base change of $f_E$ from $\GL_{2/\mathbf{Q}}$ to $\GL_{2/F}$, one has that the quotient:
\[
\frac{D_F^{-1/2} (\Omega^w_{f_E})^2}{\Omega^{w,w}_{\mathbf{f}_E}}
\]
belongs to the field extension $\mathbf{Q}(f_E)$ of $\mathbf{Q}$ generated by the Hecke eigenvalues of $f_E$, which is just $\mathbf{Q}$ since the Hecke eigenvalues of $f_E$ are rational numbers ($f_E$ corresponds to the elliptic curve $E/\mathbf{Q}$). Thus for any $\delta$ as above, the expression:
\begin{eqnarray*} 
 \frac{D_F^{1/2} \mathcal{N}_{F/\mathbf{Q}} \mathfrak{c}_{\delta} L(1,E/F,\delta) }{ \tau(\delta^{-1}) (\Omega_{f_E}^w)^2}
\end{eqnarray*}
belongs to $\mathbf{Q}(\delta)$. So in particular, if $\delta$ is quadratic, then $\delta^{-1} = \delta$ and $\mathbf{Q}(\delta) = \mathbf{Q}$, and the proposition is proved. 
\end{proof}

\bigskip

We now impose additional conditions on $\delta$. For each finite place $v$ of $F$, we fix a local uniformizer $\pi_v$ for the completion $F_v$ of $F$ at $v$. We now consider quadratic Hecke characters $\delta = \otimes^{\prime}_v \delta_v$ of $F$, unramified at primes of $\mathcal{O}_F$ dividing $N$, such that the following conditions are satisfied:
\begin{itemize}
\item (a) $\delta_v(-1) = \chi_{F,v}(-1) =w$ for archimedean $v$.
\item (b) $\delta_{\mathfrak{l}}(\pi_{\mathfrak{l}}) = \chi_{F,\mathfrak{l}}(\pi_{\mathfrak{l}})$ for all prime $\mathfrak{l}$ of $F$ dividing $M$. 
\item (c) $\delta_{\mathfrak{q}}(\pi_{\mathfrak{q}}) = - \chi_{F,\mathfrak{q}}(\pi_{\mathfrak{q}}) = -1$, for all primes $\mathfrak{q}$ of $F$ dividing $\mathcal{Q}$.   
\end{itemize}

For any such $\delta$, we claim that the sign of the functional equation for $L(s,E/F,\delta)$ is equal to $+1$. Indeed if $\mathcal{Q}$ is a product of an odd number of primes, then the sign of the functional equation for $L(s,E/F,\chi_F)$ is $-1$, while the sign of the functional equation for $L(s,E/F,\delta)$ is opposite to that of $L(s,E/F,\chi_F)$, hence is equal to $+1$. On the other hand, if $\mathcal{Q}$ is a product of an even number of primes, then the sign of the functional equation for $L(s,E/F,\chi_F)$ is $+1$, while the sign of the functional equation for $L(s,E/F,\delta)$ is equal to that of $L(s,E/F,\chi_F)$, hence is also equal to $+1$.

\bigskip

Thus by the main result of [FH], there exists infinitely many quadratic Hecke characters $\delta$ of $F$ unramified at primes dividing $N$ satisfying conditions (a), (b), (c) and such that $L(1,E/F,\delta) = L(1,\mathbf{f}_E,\delta) \neq 0$.

\bigskip

We now focus on the case where $\mathcal{Q}$ is a product of an {\bf odd} number of primes, so in particular $\mathcal{Q} \neq 1$ (recall that in the setting of sections 1 and 2, $\mathcal{Q} =p$ is an odd prime). In this case, the sign of the functional equation for $L(s,E/\mathbf{Q},\chi_1)$ and $L(s,E/\mathbf{Q},\chi_2)$ are opposite to each other, and that the sign of the functional equation for $L(s,E/F,\chi_F)$ is equal to $-1$. Thus $L^{\prime}(1,E/F,\chi_F) \neq 0$ if and only if exactly one of the following holds: 
\[
L^{\prime}(1,E/\mathbf{Q},\chi_1) \neq 0, L(1,E/\mathbf{Q},\chi_2) \neq 0
\]
or
\[
L(1,E/\mathbf{Q},\chi_1) \neq 0, L^{\prime}(1,E/\mathbf{Q},\chi_2) \neq 0.
\]

\bigskip

We now state the main result of this section, whose proof will be completed in section 4:

\begin{theorem}
With notations as above we are in the setting where $\mathcal{Q}$ is a product of an {\bf odd} number of primes. Suppose that $L^{\prime}(1,E/F,\chi_F) \neq 0$. Then for any quadratic Hecke character $\delta$ of $F$ that is unramified at primes dividing $N$, and satisfying conditions (a), (b), (c) as above, the expression :
\begin{eqnarray}
2 \frac{D_F^{1/2} \mathcal{N}_{F/\mathbf{Q}} \mathfrak{c}_{\delta} L(1,E/F,\delta) }{ \tau(\delta) (\Omega_{f_E}^w)^2}
\end{eqnarray}
is the square of a rational number.
\end{theorem}

Note that as the real and imaginary period factors $\Omega_{f_E}^{\pm}$ are well-defined up to $\mathbf{Q}^{\times}$-multiples, we have that the statement of Theorem 3.2 is meaningful (since $\Omega_{f_E}^w$ enters into the formula (3.1) as $(\Omega_{f_E}^w)^2$). 

\bigskip

Before we begin the proof of Theorem 3.2, we observe that this is exactly the situation of equation (2.2) in section 2, namely in the situation there we had $L^{\prime}(1,E/\mathbf{Q},\chi_1)\neq 0$ and $L(1,E/\mathbf{Q},\chi_2) \neq 0$, and so $L^{\prime}(1,E/F,\chi_F) =  L^{\prime}(1,E/\mathbf{Q},\chi_1) \cdot L(1,E/\mathbf{Q},\chi_2)\neq 0$.

\bigskip

\begin{remark}
As all the primes $q$ dividing $\mathcal{Q}$ are inert in $F$, and in particular unramified in $F$, we can take, for the prime $\mathfrak{q} = q \mathcal{O}_F$ of $\mathcal{O}_F$ that lies over $q$, the uniformizer $\pi_{\mathfrak{q}}$ to be simply $q$ itself. Thus by condition (c) imposed on $\delta$, we have $\delta_{\mathfrak{q}}(q) = -1$. In particular $\delta  |_{\mathbf{A}_{\mathbf{Q}}^{\times}}$ is not trivial. Thus Waldpsurger type central value formula could not be directly applied to the pair $f_E$ and $\delta$ to study the rationality property of the $L$-value $L(1,E/F,\delta)$. 
\end{remark}

\bigskip

We now begin the proof of Theorem 3.2, the argument is a generalization of that of section 3 of [M2] (and also [M3] for corrections of some of the computations in section 3 of [M2]); the basic idea is to construct auxiliary $CM$ extensions of both $\mathbf{Q}$ and $F$, and then apply Gross-Zagier type formulas. Thus we assume that the condition $L^{\prime}(1,E/F,\chi_F) \neq 0$ is satisfied. By symmetry between $\chi_1$ and $\chi_2$, we may assume that $L^{\prime}(1,E/\mathbf{Q},\chi_1) \neq 0, L(1,E/\mathbf{Q},\chi_2) \neq 0$, in which case we have:
\begin{eqnarray}
L^{\prime}(1,E/F,\chi_F) = L^{\prime}(1,E/\mathbf{Q},\chi_1) \cdot L(1,E/\mathbf{Q},\chi_2) \neq 0. 
\end{eqnarray}
The sign of the functional equation for $L(s,E/\mathbf{Q},\chi_1)$ is thus $-1$, and the sign of the functional equation for $L(s,E/\mathbf{Q},\chi_2)$ is $+1$. Fix a quadratic Hecke character $\delta$ of $F$ unramified at primes dividing $N$ and satisfying conditions (a), (b) (c) above. We may assume that $L(1,E/F,\delta) \neq 0$ (for otherwise the expression (3.1) is trivially the square of a rational number).

\bigskip

We consider the set of quadratic Dirichlet characters $\chi_3$, with conductor relatively prime to $N$, and also relatively prime to the conductor of that of $\chi_1$ and $\chi_2$, and satisfies the following conditions:
\begin{itemize}
\item $\chi_3(-1) = - \chi_1(-1) = -\chi_2(-1)$. 
\item $\chi_3(l) = \chi_1(l) = \chi_2(l)$ for all primes $l$ dividing $M$
\item $\chi_3(q) = \chi_1(q) = -\chi_2(q)$ for all primes $q$ dividing $\mathcal{Q}$. 
\end{itemize}

For any such $\chi_3$, the sign of the functional equation for $L(s,E/\mathbf{Q},\chi_3)$ is opposite to that of $L(s,E/\mathbf{Q},\chi_1)$, hence is equal to $+1$. Thus by [MM] or [FH], we may choose such a $\chi_3$ such that $L(1,E/\mathbf{Q},\chi_3) \neq 0$. Fix such a $\chi_3$, which corresponds to a quadratic extension of $\mathbf{Q}$ whose discriminant would be denoted as $D_3$. 

\bigskip

The product $\chi_1 \cdot \chi_3$ is then a quadratic Dirichlet character, corresponding to an imaginary quadratic extension $L/\mathbf{Q}$, such that all primes dividing $N=M\mathcal{Q}$ split in $L$. Similarly, the product $\chi_2 \cdot \chi_3$ is a quadratic Dirichlet character, corresponding to an imaginary quadratic extension $\widetilde{L}/\mathbf{Q}$, such that all primes dividing $M$ split in $L$, while all primes dividing $\mathcal{Q}$ are inert in $\widetilde{L}$. In addition, the pair $(\chi_1,\chi_3)$ defines the quadratic (or trivial) genus class character $\chi_{L}$ of $L$, and similarly the pair $(\chi_2,\chi_3)$ defines the quadratic (or trivial) genus class character $\chi_{\widetilde{L}}$ of $\widetilde{L}$. We also denote by $H_{\chi_L}$ the quadratic (or trivial) genus class field extension of $L$ that corresponds to $\chi_L$. 

\bigskip

We have:
\begin{eqnarray}
& & L^{\prime}(1, E/L,\chi_L) = L^{\prime}(1,E/\mathbf{Q},\chi_1) \cdot L(1,E/\mathbf{Q},\chi_3) \neq 0. \\
& & L(1,E/\widetilde{L},\chi_{\widetilde{L}} )= L(1,E/\mathbf{Q},\chi_2) \cdot L(1,E/\mathbf{Q},\chi_3) \neq 0.  
\end{eqnarray}

\bigskip
Next we consider the set of quadratic Hecke characters $\delta_1 = \otimes^{\prime}_v \delta_{1,v}$ of $F$, that are unramified at the primes of $\mathcal{O}_F$ dividing $N$ and also the at the primes of $\mathcal{O}_F$ dividing the conductor of $\delta$, and satisfying the following conditions:
\begin{itemize}
\item $\delta_{1,v}(-1) = - \delta_v(-1)$ ($=-\chi_{F,v}(-1)$) $= -w$ for archimedean $v$. 
\item $\delta_{1,\mathfrak{l}} (\pi_{\mathfrak{l}}) =\delta_{\mathfrak{l}}(\pi_{\mathfrak{l}})$ ($=\chi_{F,\mathfrak{l}}(\pi_{\mathfrak{l}})$) for all primes $\mathfrak{l}$ dividing $M$.
\item $\delta_{1,\mathfrak{q}}(\pi_{\mathfrak{q}}) = \delta_{\mathfrak{q}}(\pi_{\mathfrak{q}})$ ($=-\chi_{F,\mathfrak{q}}(\pi_{\mathfrak{q}})$) $= -1$ for all primes $\mathfrak{q}$ dividing $\mathcal{Q}$. 

\end{itemize}

For any such $\delta_1$ satisfying the above conditions, the sign of the functional equation for $L(s,E/F,\delta_1)$ is equal to that of $L(s,E/F,\delta)$ (noting that the number of archiemdean places of $F$ is equal to two), hence is equal to $+1$. Thus by [FH], there exists infinitely many such $\delta_1$ such that $L(1,E/F,\delta_1) \neq 0$. Fix such a $\delta_1$. Denote by $K_1$ the quadratic extension of $F$ that corresponds to $\delta_1$.

\bigskip

The product $\chi_F \cdot \delta_1$ is then a quadratic Hecke character of $F$, that corresponds to an imaginary quadratic $CM$ extension $K$ of $F$, such that all primes of $\mathcal{O}_F$ that divide $M$ split in $K$, while all the primes of $\mathcal{O}_F$ that divide $\mathcal{Q}$ are inert in $K$.

\bigskip
Similarly the product $\delta \cdot \delta_1$ is a quadratic Hecke character of $F$, that corresponds to an imaginary quadratic $CM$ extension $\widetilde{K}$ of $F$, such that all primes of $\mathcal{O}_F$ that divides $N=M \mathcal{Q}$ split in $\widetilde{K}$.

\bigskip

The pair $(\chi_F,\delta_1)$ then defines a quadratic (or trivial) genus class character $\delta_K$ of $K$; denote by $H_{\delta_K}$ the quadratic (or trivial) genus class field extension of $K$ that corresponds to $\delta_K$. We have:
\[
L(s,E/K,\delta_K) = L(s,E/F,\chi_F) \cdot L(s,E/F,\delta_1).
\]

Similarly the pair $(\delta,\delta_1)$ defines a quadratic genus class character $\delta_{\widetilde{K}}$ of $\widetilde{K}$; we have:
\[
L(s,E/\widetilde{K},\delta_{\widetilde{K}}) = L(s,E/F,\delta) \cdot L(s,E/F,\delta_1).
\]

In particular we have:
\begin{eqnarray}
& & L^{\prime}(1,E/K,\delta_K)= L^{\prime}(1,E/F,\chi_F) \cdot L(1,E/F,\delta_1) \neq 0 \\
& & L(1,E/\widetilde{K},\delta_{\widetilde{K}}) = L(1,E/F,\delta) \cdot L(1,E/F,\delta_1) \neq 0.
\end{eqnarray}

The equations (3.2) - (3.6) together then gives:
\begin{eqnarray}
& & L(1,E/F,\delta) \\
&=&  \frac{ L^{\prime}(1,E/L,\chi_L)  \cdot L(1,E/\widetilde{K},\delta_{\widetilde{K}} ) \cdot L(1,E/\mathbf{Q},\chi_2)^2}{L^{\prime}(1,E/K,\delta_K) \cdot L(1,E/\widetilde{L}, \chi_{\widetilde{L}} )  } \nonumber
\end{eqnarray}

\bigskip
To analyze the various terms in (3.7), we firstly note that:
\begin{eqnarray}
\frac{c_{\chi_2}  L(1,E/\mathbf{Q},\chi_2)}{\tau(\chi_2) \Omega_{f_E}^w} \in \mathbf{Q}^{\times}
\end{eqnarray}

\bigskip
As for $ L^{\prime}(1,E/L,\chi_L)$, since $L$ is an imaginary quadratic extension of $\mathbf{Q}$ such that all primes dividing $N$ splits in $L$, we can apply the Gross-Zagier formula to the pair $E/\mathbf{Q}$ and $\chi_L$ [GZ]:

\begin{eqnarray}
L^{\prime}(1,E/L,\chi_L) = \frac{4}{D_L^{1/2}} \frac{\langle f_E,f_E \rangle}{\deg_{E/\mathbf{Q}}} \height(\mathbf{P}_{\chi_L})
\end{eqnarray}

The explanation of the terms involved are as follows. Here $D_L$ is the absolute value of the discriminant of $L$, while $\langle f_E,f_E \rangle$ is the Petersson inner product of $f_E$, normalized as on p. 899 of [M1] (with $F=\mathbf{Q}$ and hence $d=1$ in {\it loc. cit.}).

\bigskip
We fix a modular parametrization over $\mathbf{Q}$ of $E/\mathbf{Q}$ by the modular curve $X_0(N)/\mathbf{Q}$, and denote by $\deg_{E/\mathbf{Q}}$ the degree of this modular parametrization (the degrees of all the different modular parametrization of $E/\mathbf{Q}$ by $X_0(N)/\mathbf{Q}$ differ at most by multiplication of squares of rational numbers). Finally $\mathbf{P}_{\chi_L} \in (E(H_{\chi_L}) \otimes \mathbf{Q})^{\chi_L}$ and $\height (\mathbf{P}_{\chi_L})$ is its Neron-Tate height. 
\bigskip

\bigskip
Since $L^{\prime}(1,E/L,\chi_L) =  L^{\prime}(1,E/\mathbf{Q},\chi_1) \cdot L(1,E/\mathbf{Q},\chi_3) \neq 0$, the point $\mathbf{P}_{\chi_L}$ has non-zero Neron-Tate height and is hence non-torsion. The theorem of Kolyvagin [K] asserts that $\dim_{\mathbf{Q}} (E(H_{\chi_L}) \otimes \mathbf{Q})^{\chi_L}=1$ and that $(E(H_{\chi_L}) \otimes \mathbf{Q})^{\chi_L}$ is spanned by $\mathbf{P}_{\chi_L}$. 

\bigskip
More precisely, since $L^{\prime}(1,E/\mathbf{Q},\chi_1) \neq 0$ and $L(1,E/\mathbf{Q},\chi_3) \neq 0$, Kolyvagin's theorem [K] asserts that $\dim_{\mathbf{Q}} (E(\mathbf{Q}(\sqrt{D_1}) )  \otimes \mathbf{Q})^{\chi_1} =1$ and $\dim_{\mathbf{Q}} (E(\mathbf{Q}(\sqrt{D_3}) )  \otimes \mathbf{Q})^{\chi_3} =0$. By the same argument as in Theorem 4.7 of [BD2] or Corollary 4.2 of [M1], we see that the point $\mathbf{P}_{\chi_L}$ may be chosen so that $\mathbf{P}_{\chi_L} \in (E(\mathbf{Q}(\sqrt{D_1}) ) \otimes \mathbf{Q})^{\chi_1}$; in fact the vanishing of $(E(\mathbf{Q}(\sqrt{D_3})  ) \otimes \mathbf{Q})^{\chi_3}$ gives $(E(H_{\chi_L}) \otimes \mathbf{Q})^{\chi_L} = (E(\mathbf{Q}(\sqrt{D_1})  ) \otimes \mathbf{Q})^{\chi_1}$.

\bigskip

As for $L^{\prime}(1,E/K,\delta_K)$, we would like to use the generalized Gross-Zagier formula of Zhang [Z1,Z2] to the pair $E/F$ and $\delta_K$. Note that the conductor of $E/F$ is equal to $N \mathcal{O}_F = M\mathcal{O}_F \cdot \mathcal{Q} \mathcal{O}_F$, while $K$ is an imaginary quadratic $CM$ extension of the real quadratic field $F$, such that all primes of $\mathcal{O}_F$ dividing $M \mathcal{O}_F$ splits in $K$, and all primes of $\mathcal{O}_F$ dividing $\mathcal{Q} \mathcal{O}_F$ are inert in $K$. Finally note that $\mathcal{Q} \mathcal{O}_F$ is a product of an odd number of distinct prime ideals of $\mathcal{O}_F$ (recall that $\mathcal{Q}$ is a product of an odd number of primes and that all primes dividing $\mathcal{Q}$ are inert in $F$). Thus we can apply the generalized Gross-Zagier formula of Zhang [Z1,Z2] to the pair $E/F$ and $\delta_K$ and obtain:
\begin{eqnarray}
L^{\prime}(1,E/K,\delta_K) = \frac{8}{(\mathcal{N}_{F/\mathbf{Q}}   D_{K/F})^{1/2}} \frac{\langle \mathbf{f}_E,\mathbf{f}_E \rangle }{\deg_{E/F}} \height(\mathbf{P}_{\delta_K})
\end{eqnarray}

\bigskip
The explanation of the terms involved are as follows. Here $D_{K/F}$ is the relative discriminant ideal of $K/F$ and $\mathcal{N}_{F/\mathbf{Q}} D_{K/F}$ is its norm, while $\langle \mathbf{f}_E,\mathbf{f}_E \rangle$ is the Petersson inner product of $\mathbf{f}_E$, normalized as on p. 899 of [M1]. 

\bigskip

We fix a modular parametrization over $F$ of $E/F$ by the Shimura curve $X(M\mathcal{O}_F,\mathcal{Q} \mathcal{O}_F)/F$. Here $X(M\mathcal{O}_F,\mathcal{Q} \mathcal{O}_F)/F$ is the Shimura curve over $F$ with Eichler level $M\mathcal{O}_F$ associated to the quaternion algebra $\mathcal{B}$ over the real quadratic field $F$ that ramifies exactly at the primes of $\mathcal{O}_F$ dividing $\mathcal{Q} \mathcal{O}_F$ and at one of the two archimedean places of $F$ (if we fix an embedding $\iota: F \hookrightarrow \mathbf{R}$, then the archimedean place of $F$ where $\mathcal{B}$ ramifies corresponds to the embedding $\iota \circ \sigma: F \hookrightarrow \mathbf{R}$, where $\sigma$ is the Galois conjugation of $F$ over $\mathbf{Q}$); {\it c.f.} for example section 4.1 of [M1] for the Shimura curves that we use. Denote by $\deg_{E/F}$ the degree of this modular parametrization (the degrees of all the different modular parametrization of $E/F$ by $X(M\mathcal{O}_F,\mathcal{Q} \mathcal{O}_F)/F$ differ at most by multiplication of squares of rational numbers); we remark that the degree $\deg_{E/F}$ of the modular parametrization of $E/F$ by $X(M\mathcal{O}_F,\mathcal{Q} \mathcal{O}_F)/F$ is in general defined in terms of the Jacobian of $X(M\mathcal{O}_F,\mathcal{Q} \mathcal{O}_F)/F$; we will discuss this in more detail in section 4 below. Finally $\mathbf{P}_{\delta_K} \in (E(H_{\delta_K}) \otimes \mathbf{Q})^{\delta_K}$ and $\height (\mathbf{P}_{\delta_K})$ is its Neron-Tate height.

\bigskip
In a way similar to before, since we have:
\[
L^{\prime}(1,E/K,\delta_K) =  L^{\prime}(1,E/F,\chi_F) \cdot L(1,E/F,\delta_1) \neq 0,
\]
one has that the point $\mathbf{P}_{\delta_K}$ has non-zero Neron-Tate height and is hence non-torsion. Theorem A of [Z1] (following the methods of Kolyvagin-Logachev [KL]) asserts that one has $\dim_{\mathbf{Q}} (E(H_{\delta_K}) \otimes \mathbf{Q})^{\delta_K}=1$ and that $(E(H_{\delta_K}) \otimes \mathbf{Q})^{\delta_K}$ is spanned by $\mathbf{P}_{\delta_K}$.

\bigskip

More precisely, since $L^{\prime}(1,E/F,\chi_F) \neq 0$ and $L(1,E/F,\delta_1) \neq 0$, Theorem A of [Z1] asserts that one has 
\[
\dim_{\mathbf{Q}} (E(H_{\chi_F})    \otimes \mathbf{Q})^{\chi_F} = 1 \mbox{ and } \dim_{\mathbf{Q}} (E(K_1)   \otimes \mathbf{Q})^{\delta_1} =0
\] 
(recall that $K_1$ is the quadratic extension of $F$ that corresponds to $\delta_1$). By the same argument as in Corollary 4.2 of [M1], we see that the point $\mathbf{P}_{\delta_K}$ may be chosen so that $\mathbf{P}_{\delta_K} \in (E(H_{\chi_F}) \otimes \mathbf{Q})^{\chi_F}$; in fact the vanishing of $(E(K_1)   \otimes \mathbf{Q})^{\delta_1}$ gives $(E(H_{\delta_K}) \otimes \mathbf{Q})^{\delta_K} = (E(H_{\chi_F})   \otimes \mathbf{Q})^{\chi_F}$.

\bigskip

We now apply the same argument back to $(E(H_{\chi_F}) \otimes \mathbf{Q})^{\chi_F}$, using:
\[
L^{\prime}(1,E/F,\chi_F) = L^{\prime}(1,E/\mathbf{Q},\chi_1) \cdot L(1,E/\mathbf{Q},\chi_2) \neq 0.
\]
Namely that since $L(1,E/\mathbf{Q},\chi_2) \neq 0$, Kolyvagin's theorem [K] asserts that $\dim_{\mathbf{Q}} (E(\mathbf{Q}(\sqrt{D_2}) )  \otimes \mathbf{Q})^{\chi_2} =0$. By the same argument as in Theorem 4.7 of [BD2] or Corollary 4.2 of [M1], we see that $(E(H_{\chi_F}) \otimes \mathbf{Q})^{\chi_F} = (E(\mathbf{Q}(\sqrt{D_1}) )  \otimes \mathbf{Q})^{\chi_1}$. Thus the point $\mathbf{P}_{\delta_K}$ may also be taken as a non-torsion point in $(E(\mathbf{Q}(\sqrt{D_1}) )  \otimes \mathbf{Q})^{\chi_1}$.

\bigskip

Thus both the points $\mathbf{P}_{\chi_L}$ and $\mathbf{P}_{\delta_K}$ can be taken as non-torsion points in $(E(\mathbf{Q}(\sqrt{D_1}))\otimes \mathbf{Q}  )^{\chi_1}$, which is of dimension one over $\mathbf{Q}$. As the Neron-Tate height is a quadratic form, we have 
\[
\frac{\height (\mathbf{P}_{\chi_L}) }{\height (\mathbf{P}_{\delta_K})} \in (\mathbf{Q}^{\times})^2
\] 
Combining this with equation (3.9) and (3.10) gives:
\begin{eqnarray}
& & \frac{L^{\prime}(1,E/L,\chi_L)}{L^{\prime}(1,E/K,\delta_K)} \\
&=& 2 \cdot \big( \frac{\mathcal{N}_{F/\mathbf{Q}}  D_{K/F} }{D_L} \big)^{1/2} \cdot \frac{\deg_{E/F}}{\deg_{E/\mathbf{Q}}} \cdot \frac{\langle f_E,f_E \rangle}{\langle \mathbf{f}_E,\mathbf{f}_E \rangle} \bmod{(\mathbf{Q}^{\times})^2} \nonumber
\end{eqnarray}

\bigskip

Next we deal with the value $L(1,E/\widetilde{L},\chi_{\widetilde{L}})$. Firstly let $B$ be the definite quaternion algebra over $\mathbf{Q}$ that ramifies at the primes dividing $\mathcal{Q}$ and at the archimedean place of $\mathbf{Q}$. We denote by $\phi_E$ the scalar-valued automorphic eigenform (with trivial central character) with respect to the group $B^{\times}$, with Eichler level $M$, that corresponds to $f_E$ under the Jacquet-Langlands correspondence; we normalize $\phi_E$ be requiring that the values taken by the automorphic form $\phi_E$ lies in $\mathbf{Q}$ (which is possible because the Hecke eigenvalues of $f_E$ are in $\mathbf{Z}$). With this normalization $\phi_E$ is then well-defined up to $\mathbf{Q}^{\times}$-multiples. Fix such a choice for $\phi_E$ in what follows. We denote by $\langle \phi_E,\phi_E \rangle$ the Petersson inner product of $\phi_E$, which is normalized as on p. 901 of [M1]; we will explicate this normalization in more details in section 4.

\bigskip

We would like to apply the central value formula of Zhang [Z2] (that generalizes the results in [Gr]) to the $L$-value $L(1,E/\widetilde{L},\chi_{\widetilde{L}})$. Indeed $\widetilde{L}$ is an imaginary quadratic extension of $\mathbf{Q}$ such that all primes dividing $M$ splits in $\widetilde{L}$, and all primes dividing $\mathcal{Q}$ are inert in $\widetilde{L}$. Thus Zhang''s central value formula [Z2] is applicable to $L(1,E/\widetilde{L},\chi_{\widetilde{L}})$; and under our normalization that the values taken by $\phi_E$ lie in $\mathbf{Q}$, the formula of Zhang gives ({\it c.f.} also equation (3.9) and (3.10) of [M1], with $F=\mathbf{Q}$ and $d=1$ in {\it loc. cit.}):
\begin{eqnarray}
& & L(1,E/\widetilde{L},\chi_{\widetilde{L}}) \\
&=& \frac{1}{D_{\widetilde{L}}^{1/2}} \frac{\langle f_E,f_E \rangle}{\langle \phi_E,\phi_E \rangle} \times \mbox{ the square of a rational number}. \nonumber
\end{eqnarray} 
As before $D_{\widetilde{L}}$ as the absolute value of the discriminant of $\widetilde{L}$. 

\bigskip
Now we turn to the value $L(1,E/\widetilde{K},\delta_{\widetilde{K}})$. Firstly we define $B_F :=  B \otimes_{\mathbf{Q}} F$, where $B$ is as before. Then $B_F$ is the totally definite quaternion algebra over $F$, that splits at all the finite places of $F$, and is ramified at all the archimedean places of $F$. Denote by $\Phi_E$ the scalar-valued automorphic eigenform (with trivial central character) with respect to the group $B_F^{\times}$, with Eichler level $N\mathcal{O}_F = M \mathcal{Q} \mathcal{O}_F$, that corresponds to $\mathbf{f}_E$ under the Jacquet-Langlands correspondence; we normalize $\Phi_E$ be requiring that the values taken by the automorphic form $\Phi_E$ lie in $\mathbf{Q}$ (which is possible because the Hecke eigenvalues of $\mathbf{f}_E$ are in $\mathbf{Z}$). With this normalization $\Phi_E$ is then well-defined up to $\mathbf{Q}^{\times}$-multiples. Fix such a choice for $\Phi_E$ in what follows. We denote by $\langle \Phi_E,\Phi_E \rangle$ the Petersson inner product of $\Phi_E$, which is normalized as on p. 901 of [M1]; we will explicate this normalization in more details in section 4.

\bigskip

We again would like to apply the central value formula of Zhang [Z2], to the $L$-value $L(1,E/\widetilde{L},\delta_{\widetilde{K}})$. For this recall that $\widetilde{K}$ is an imaginary quadratic $CM$ extension of $F$ such that all primes of $\mathcal{O}_F$ dividing $N \mathcal{O}_F = M \mathcal{Q} \mathcal{O}_F$ split in $\widetilde{K}$. Thus the central value formula of Zhang [Z2] is applicable to $L(1,E/\widetilde{K},\delta_{\widetilde{K}})$; and under our normalization that the values taken by $\Phi_E$ lie in $\mathbf{Q}$, the formula of Zhang gives ({\it c.f.} also equation (3.9) and (3.10) of [M1]):
\begin{eqnarray}
& & L(1,E/\widetilde{K},\delta_{\widetilde{K}}) \\
&=& \frac{1}{(\mathcal{N}_{F/\mathbf{Q} }D_{\widetilde{K}/F})^{1/2}} \frac{\langle \mathbf{f}_E,\mathbf{f}_E \rangle}{\langle \Phi_E,\Phi_E \rangle} \times \mbox{ the square of a rational number}. \nonumber
\end{eqnarray} 
As before $D_{\widetilde{K}/F}$ is the relative discriminant ideal of $\widetilde{K}$ over $F$, and $\mathcal{N}_{F/\mathbf{Q}}D_{\widetilde{K}/F}$ is its norm.

\bigskip

Now the values $L(1,E/\widetilde{L},\chi_{\widetilde{L}})$ and $L(1,E/\widetilde{K},\delta_{\widetilde{K}})$ are both non-zero. Thus (3.12) and (3.13) gives:
\begin{eqnarray}
& & \frac{L(1,E/\widetilde{K},\delta_{\widetilde{K}})}{L(1,E/\widetilde{L},\chi_{\widetilde{L}})}\\
& =& \big( \frac{D_{\widetilde{L}}}{\mathcal{N}_{F/\mathbf{Q}} D_{\widetilde{K}/F}} \big)^{1/2} \frac{\langle \mathbf{f}_E,\mathbf{f}_E \rangle}{\langle f_E,f_E \rangle}  \frac{\langle \phi_E,\phi_E \rangle}{\langle \Phi_E,\Phi_E \rangle}\bmod{(\mathbf{Q}^{\times})^2} \nonumber
\end{eqnarray}

\bigskip

Substituting (3.8), (3.11) and (3.14) into (3.7), we obtain:
\begin{eqnarray}
& & \frac{L(1,E/F,\delta)}{(\Omega^w_{f_E})^2} \\
&=& 2  \cdot \tau(\chi_2)^2  \big( \frac{D_{\widetilde{L}}}{D_L} \big)^{1/2} \big( \frac{\mathcal{N}_{F/\mathbf{Q}}D_{K/F}  }{\mathcal{N}_{F/\mathbf{Q}}D_{\widetilde{K}/F } }  \big)^{1/2}\frac{\deg_{E/F}}{\deg_{E/\mathbf{Q}}} \frac{\langle \phi_E,\phi_E \rangle}{\langle \Phi_E,\Phi_E \rangle}  \bmod{(\mathbf{Q}^{\times})^2} \nonumber
\end{eqnarray}

\bigskip

Now we have the standard conductor-discriminant identities:
\begin{eqnarray*}
& & D_F = c_{\chi_1} \cdot c_{\chi_2}, \,\  D_L = c_{\chi_1} \cdot c_{\chi_3}, \,\ D_{\widetilde{L}} = c_{\chi_2} \cdot c_{\chi_3} \\
& & \mathcal{N}_{F/\mathbf{Q}} D_{\widetilde{K}/F} = \mathcal{N}_{F/\mathbf{Q}} \mathfrak{c}_{\delta} \cdot \mathcal{N}_{F/\mathbf{Q}} \mathfrak{c}_{\delta_1} \\
& &  \mathcal{N}_{F/\mathbf{Q}} D_{K/F} =   \mathcal{N}_{F/\mathbf{Q}} \mathfrak{c}_{\chi_F} \cdot \mathcal{N}_{F/\mathbf{Q}} \mathfrak{c}_{\delta_1} =   \mathcal{N}_{F/\mathbf{Q}} \mathfrak{c}_{\delta_1}
\end{eqnarray*}
(for the identity on the last line, note that $\chi_F$ is a narrow genus class character of $F$ and so $\mathfrak{c}_{\chi_F} =\mathcal{O}_F$). Substituting these identities into (3.15), we obtain:
\begin{eqnarray}
& & \frac{D_F^{1/2} ( \mathcal{N}_{F/\mathbf{Q}} \mathfrak{c}_{\delta})^{1/2} L(1,E/F,\delta)}{(\Omega^w_{f_E})^2} \\
&=& 2 \cdot  c_{\chi_2} \cdot  \tau(\chi_2)^2 \cdot  \frac{\deg_{E/F}}{\deg_{E/\mathbf{Q}}} \frac{\langle \phi_E,\phi_E \rangle}{\langle \Phi_E,\Phi_E \rangle}  \bmod{(\mathbf{Q}^{\times})^2} \nonumber
\end{eqnarray}

\bigskip
And since the characters $\chi_2$ and $\delta$ are quadratic, we have the Gauss sum identities:
\begin{eqnarray*}
& & \tau(\chi_2)^2 = \chi_2(-1) \cdot c_{\chi_2} = w \cdot c_{\chi_2} \\
& & \tau(\delta) = i^{m_{\delta} }(\mathcal{N}_{F/\mathbf{Q}} \mathfrak{c}_{\delta})^{1/2} = w \cdot (\mathcal{N}_{F/\mathbf{Q}} \mathfrak{c}_{\delta})^{1/2}
\end{eqnarray*}
here $m_{\delta}$ is equal to the number of (real) archimedean places $v$ of $F$ such that $\delta_v(-1) = -1$; recall also that in our case $\delta_v(-1) = \chi_{F,v}(-1) =w$ for archimedean place $v$ of $F$, hence $m_{\delta} =0$ is $w=+1$, and $m_{\delta}=2$ if $w=-1$; thus $i^{m_{\delta}}=w$.

\bigskip

Thus applying these identities to (3.16), we finally obtain:
\begin{eqnarray}
& & 2 \frac{D_F^{1/2}  \mathcal{N}_{F/\mathbf{Q}} \mathfrak{c}_{\delta} L(1,E/F,\delta)}{\tau(\delta) (\Omega^w_{f_E})^2} \\
&=&  \frac{\deg_{E/F}}{\deg_{E/\mathbf{Q}}} \frac{\langle \phi_E,\phi_E \rangle}{\langle \Phi_E,\Phi_E \rangle}  \bmod{(\mathbf{Q}^{\times})^2} \nonumber
\end{eqnarray}

\bigskip

On comparing the expression (3.1) with (3.17) we see that the proof of Theorem 3.2 is completed, once we establish that the right hand side of (3.17) is equal to the square of a rational number. This is accomplished by the following key lemma, which by itself is independent of the discussion of $L$-values in this section:

\begin{lemma}

Let $E/\mathbf{Q}$ be an elliptic curve, whose conductor $N$ is of the form $N=M  \mathcal{Q}$, where $(M,\mathcal{Q})=1$ and $\mathcal{Q}$ is a product of an {\bf odd} number of distinct primes. 

\bigskip

Denote by $f_E$ the normalized weight two cuspidal eigenform of level $N=M \mathcal{Q}$ that corresponds to $E/\mathbf{Q}$. Denote by $\deg_{E/\mathbf{Q}}$ the degree of a (fixed) parametrization of $E/\mathbf{Q}$ by the modular curve $X_0(N)/\mathbf{Q}$.

\bigskip

Let $F$ be a real quadratic extension of $\mathbf{Q}$, such that all primes dividing $M$ split in $F$, while all primes dividing $\mathcal{Q}$ are inert in $F$. Let $\mathbf{f}_E$ be the base change of $f_E$ from $\GL_{2/\mathbf{Q}}$ to $\GL_{2/F}$; thus $\mathbf{f}_E$ is a parallel weight two cuspidal Hilbert eigenform over $F$ of level $N \mathcal{O}_F=M Q \mathcal{O}_F$ that corresponds to $E/F$. Denote by $\deg_{E/F}$ the degree of a (fixed) parametrization of $E/F$ by $X(M\mathcal{O}_F,\mathcal{Q} \mathcal{O}_F)/F$, the Shimura curve of Eichler level $M \mathcal{O}_F$ associated to the quaternion algebra $\mathcal{B}$ over $F$ that ramifies exactly at the primes of $\mathcal{O}_F$ dividing $\mathcal{Q} \mathcal{O}_F$ and at one of the two archimedean places of $F$ (if we fix an embedding $\iota: F \hookrightarrow \mathbf{R}$, then the archimedean place of $F$ where $\mathcal{B}$ ramifies corresponds to the embedding $\iota \circ \sigma: F \hookrightarrow \mathbf{R}$, where $\sigma$ is the Galois conjugation of $F$ over $\mathbf{Q}$).

\bigskip
Denote by $B$ the definite quaternion algebra over $\mathbf{Q}$, that ramifies at the primes dividing $\mathcal{Q}$ and the archimedean place. Let $\phi_E$ be a scalar-valued automorphic eigenform (with trivial central character) with respect to the group $B^{\times}$ of Eichler level $M$, that corresponds to $f_E$ under the Jacquet-Langlands correspondence. Similarly let $\Phi_E$ be a scalar-valued automorphic eigenform (with trivial central character) with respect to the group $B_F^{\times}= (B \otimes_{\mathbf{Q}} F)^{\times}$ of Eichler level $M \mathcal{Q} \mathcal{O}_F$, that corresponds to $\mathbf{f}_E$ under the Jacquet-Langlands correspondence (recall that $B \otimes_{\mathbf{Q}} F$ ramifies exactly at the archimedean places of $F$ and is split at all the finite places of $F$).

\bigskip

We normalize $\phi_E$ and $\Phi_E$ so that they take values in $\mathbf{Q}$ (with this condition $\phi_E$ and $\Phi_E$ are uniquely determined up to $\mathbf{Q}^{\times}$-multiples). Denote by $\langle \phi_E,\phi_E \rangle$ the Petersson inner product of $\phi_E$ with itself, and similarly for $\langle \Phi_E,\Phi_E \rangle$ (these Petersson inners products are normalized as on p. 901 of [M2]).

\bigskip

Then we have:
\begin{eqnarray}
\deg_{E/\mathbf{Q}} = \big( \prod_{q |\mathcal{Q}} \overline{c}_q  \big) \langle \phi_E, \phi_E \rangle        \mod{(\mathbf{Q}^{\times})^2}
\end{eqnarray}
and
\begin{eqnarray}
\deg_{E/F} =    \big( \prod_{q |\mathcal{Q}} \overline{c}_q  \big)  \langle \Phi_E, \Phi_E \rangle        \mod{(\mathbf{Q}^{\times})^2}
\end{eqnarray}
where for $q | \mathcal{Q}$, we denote $\overline{c}_q = \overline{c}_q(E/\mathbf{Q})$ as the $q$-valuation of the minimal discriminant of $E/\mathbf{Q}$.

\bigskip
In particular:
\begin{eqnarray*}
\frac{\langle \Phi_E,\Phi_E \rangle  }{ \deg_{E/F }}=  \frac{\langle \phi_E,\phi_E \rangle }{  \deg_{E/\mathbf{Q}}} \mod{(\mathbf{Q}^{\times})^2}
\end{eqnarray*}
\end{lemma}

\bigskip

We will establish Lemma 3.4 in the next section, and thus completing the proof of Theorem 3.2.

\bigskip

\begin{remark}
For more discussions related to the explicit form of generalizations of Gross-Zagier type formulas, we refer to [H], specifically Theorem 4.4.2 and Theorem 5.6.2 of [H], or [CST], specifically Theorem 1.5 and Theorem 1.10 of [CST].
\end{remark}

\bigskip

Before proceeding to section 4, we make the following additional observations, which are not needed for the proof of Lemma 3.4 in section 4. Here we only require $\mathcal{Q}$ to be square-free. Let $E_{\delta}/F$ be the quadratic twist of $E/F$ by the quadratic Hecke character $\delta$ of $F$, with $\delta$ as before. We have $L(s,E_{\delta}/F)=L(s,E/F,\delta)$. Now the conductor of $E/F$, which is equal to $N \mathcal{O}_F = M \mathcal{Q} \mathcal{O}_F$, is relatively prime to the conductor $\mathfrak{c}_{\delta}$ of $\delta$. Thus the conductor of $E_{\delta}/F$ is equal to $N \mathfrak{c}_{\delta}^2$. Assume for simplicity that, in addition that $2$ is unramified in $F$. Then the rank $0$ case of the Birch and Swinnerton-Dyer conjecture for $E_{\delta}/F$ implies that the normalized $L$-value (3.1) is always the square of a rational number, up to a factor of two, without having to assume that $L^{\prime}(1,E/F,\chi_F) \neq 0$. The argument we give below is similar to that of Proposition 2.2 of [M2] and Proposition 1 of [M3].

\bigskip

We may choose $\Omega_{f_E}^{\pm}$ so that they are equal to the real and imaginary periods $\Omega^{\pm}_{E/\mathbf{Q}}$ of $E/\mathbf{Q}$ with respect to the Neron differential associated to a global minimal Weierstrass equation for $E/\mathbf{Q}$. Define:
\[
\Omega_{E_{\delta}/F} := \frac{\tau(\delta)}{\mathcal{N}_{F/\mathbf{Q}} \mathfrak{c}_{\delta}} (\Omega_{E/\mathbf{Q}}^w)^2 = \frac{1}{(\mathcal{N}_{F/\mathbf{Q}} \mathfrak{c}_{\delta})^{1/2}}  |\Omega_{E/\mathbf{Q}}^w |^2.
\]
Up to a power of two, the number $\Omega_{E_{\delta}/F}$ is equal to the real period factor for the Birch and Swinnerton-Dyer conjecture for $E_{\delta}/F$, {\it c.f} for example Corollary 2.6 and Theorem 3.2 of [P]; here we note that, because of the assumption that $2$ is unramified in $F$, the reasoning in {\it loc. cit.} concerning minimal Weierstrass equations at $2$-adic valuation, also applies to our current situation of $E_{\delta}/F$.  
\bigskip

Now we may assume that $L(1,E_{\delta}/F) = L(1,E/F,\delta) \neq 0$. Then the rank $0$ case of the Birch and Swinnerton-Dyer conjecture for $E_{\delta}/F$ gives:
\begin{eqnarray*}
& & \frac{D_F^{1/2} \mathcal{N}_{F/\mathbf{Q}} \mathfrak{c}_{\delta}  L(1,E/F,\delta)}{\tau(\delta) ( \Omega_{E/\mathbf{Q}}^w)^2} = \frac{D_F^{1/2} L(1,E_{\delta}/F)}{\Omega_{E_{\delta}/F}}   \\
&\stackrel{\cdot}{=}& \frac{ (\prod_{\mathfrak{l}} c_{\mathfrak{l}}(E_{\delta}/F)) \cdot \# \mbox{III}(E_{\delta}/F)}{(\# E_{\delta}(F)_{tors})^2}
\end{eqnarray*}
with the symbol $\stackrel{\cdot}{=}$ being equality up to a power of two (recall that the number $\Omega_{E_{\delta}/F}$ is equal to the real period factor for $E_{\delta}/F$, up to a power of two). The numbers $c_{\mathfrak{l}}(E_{\delta}/F)$ are the local Tamagawa factors of $E_{\delta}/F$ at the various primes $\mathfrak{l}$ of $\mathcal{O}_F$. We have $c_{\mathfrak{l}}(E_{\delta}/F)=1$ for $\mathfrak{l}$ not dividing the conductor of $E_{\delta}/F$; i.e. $c_{\mathfrak{l}}(E_{\delta}/F)=1$ for $\mathfrak{l}$ relatively prime to $N =M \mathcal{Q}$ and $\mathfrak{c}_{\delta}$. Finally by the Shafarevich-Tate conjecture for $E_{\delta}/F$ (which is part of the Birch and Swinnerton-Dyer conjecture), the order of the Shafarevich-Tate group $\# \mbox{III}(E_{\delta}/F)$ is a finite, in which case it must be a square. 

\bigskip
Thus it suffices to show that $\prod_{\mathfrak{l}} c_{\mathfrak{l}}(E_{\delta}/F)$ is a square, up to a factor of two. We proceed as follows. 

\bigskip
Now if $\mathfrak{l}$ is a prime of $\mathcal{O}_F$ dividing $M$, then denoting by $l$ the rational prime of $\mathbf{Z}$ that lies below $\mathfrak{l}$, we have $l$ splits in $F$, and $l \mathcal{O}_F = \mathfrak{l} \cdot \overline{ \mathfrak{l}}$, where $\overline{\mathfrak{l}}$ is the conjugate of $\mathfrak{l}$ over $\mathbf{Q}$. The local components $\delta_{\mathfrak{l}}$ and $\delta_{\overline{\mathfrak{l}}}$ are trivial, and hence:
\[
c_{\mathfrak{l}}(E_{\delta}/F) = c_{\mathfrak{l}}(E/F) = c_l(E/\mathbf{Q})
\] 
\[
c_{\overline{\mathfrak{l}}}(E_{\delta}/F) = c_{\overline{\mathfrak{l}}}(E/F) = c_l(E/\mathbf{Q})
\] 
where $c_l(E/\mathbf{Q})$ is the local Tamagawa factor of $E/\mathbf{Q}$ at the rational prime $l$. Thus $c_{\mathfrak{l}}(E_{\delta}/F) \cdot  c_{\overline{\mathfrak{l}}}(E_{\delta}/F) =( c_l(E/\mathbf{Q}))^2$. Hence:
\[
\prod_{\mathfrak{l} | M} c_{\mathfrak{l}}(E_{\delta}/F) = ( \prod_{l | M} c_l(E/\mathbf{Q}) )^2
\]
which is a square.

\bigskip

If $ \mathfrak{l}$ is a prime of $\mathcal{O}_F$ dividing $\mathcal{Q}$, then we have $\mathfrak{l} = l \mathcal{O}_F$, where $l$ is the rational prime of $\mathbf{Z}$ that lies below $\mathfrak{l}$. Now $E/\mathbf{Q}$ has multiplicative reduction at primes dividing $\mathcal{Q}$, and the rational prime $l | \mathcal{Q}$ is inert in $F$. Thus $E/F$ has {\bf split} multiplicative reduction at $\mathfrak{l}$. In turn, since the local component $\delta_{\mathfrak{l}}$ of $\delta$ satisfies $\delta_{\mathfrak{l}}(\pi_{\mathfrak{l}}) =-1$, it follows that $E_{\delta}/F$ has {\bf non-spit} multiplicative reduction at $\mathfrak{l}$. In this case, we have $c_{\mathfrak{l}}(E_{\delta}/F)$ is equal to $1$ if the normalized valuation at $\mathfrak{l}$ of the minimal discriminant of $E_{\delta}/F$ is odd, and is equal to $2$ otherwise. In any case it is a square up to a factor of $2$. 

\bigskip

Finally if $\mathfrak{l}$ divides $\mathfrak{c}_{\delta}$, then $E_{\delta}/F$ has additive reduction at $\mathfrak{l}$. Firstly assume that $\mathfrak{l}$ is relatively prime to $2$. Then we are in case 6 of Tate's algorithm [T1]; the reduction type is of $I_0^*$, and we have that $c_{\mathfrak{l}}(E_{\delta}/F)$ is equal to $1,2$ or $4$, {\it c.f.} Proposition 5 of [Ru]. In any case it is a square up to a factor of $2$. 

\bigskip

Now we deal with the case where $\mathfrak{l}$ divides $\mathfrak{c}_{\delta}$ and $\mathfrak{l}$ lies above the prime $2$. Recall that we have made the additional assumption that $2$ is unramified in $F$. Also note that the normalized valuation $\ord_{\mathfrak{l}} \mathfrak{c}_{\delta}$ of $\mathfrak{c}_{\delta}$ at $\mathfrak{l}$ is either equal to $2$ or $3$. Call these case (i) and case (ii) respectively. 

\bigskip

Define $f_{\mathfrak{l}}(E_{\delta}/F)$ to be the normalized valuation at $\mathfrak{l}$ of the conductor of $E_{\delta}/F$. As recalled above, the conductor of $E_{\delta}/F$ is equal to $N \mathfrak{c}_{\delta}^2$. Thus $f_{\mathfrak{l}}(E_{\delta}/F)$ is equal to $4$ in case (i) and $6$ in case (ii). 
\bigskip

Now we need to know something about $\ord_{\mathfrak{l}} \Delta^{min}(E_{\delta}/F)$, the normalized valuation at $\mathfrak{l}$ of the minimal discriminant ideal $\Delta^{min}(E_{\delta}/F)$ of $E_{\delta}/F$. From proposition 2.4 of [P], we see that:
\begin{eqnarray*}
\ord_{\mathfrak{l}} \Delta^{min}(E_{\delta}/F)=12, \,\   \mbox{in case (i)}
\end{eqnarray*}
and
\begin{eqnarray*}
\ord_{\mathfrak{l}} \Delta^{min}(E_{\delta}/F)= 6 \mbox{ or } 18, \,\   \mbox{in case (ii) }
\end{eqnarray*}
Specifically, case (i) corresponds to part 2 b) of Proposition 2.4 of [P], while case (ii) corresponds to part 2(c) of Proposition 2.4 of [P]. Here again, the reasoning in {\it loc. cit.} concerning minimal Weierstrass equations at $2$-adic valuation also applies to our current situation of $E_{\delta}/F$, because of our additional assumption that $2$ is unramified in $F$.

\bigskip

Now we use Ogg's formula [O]:
\[
f_{\mathfrak{l}}(E_{\delta}/F) = 1 - m_{\mathfrak{l}}(E_{\delta}/F) + \ord_{\mathfrak{l}} \Delta^{min}(E_{\delta}/F)
\]
where $m_{\mathfrak{l}}(E_{\delta}/F)$ is by definition the number of irreducible components (not counting multiplicities) of the geometric special fiber of the proper minimal regular model of $E_{\delta}/F$ over $\mathcal{O}_{F_{\mathfrak{l}}}$. It follows that in case (i) we have $m_{\mathfrak{l}}(E_{\delta}/F)=9$, while in case (ii) we have $m_{\mathfrak{l}}(E_{\delta}/F)$ is equal to $1$ or $13$. By going through the list of Kodaira-Neron for additive reduction types (for example the table in section 6 of  [T1]), we see that in case (i), the possible reduction types for $E_{\delta}/F$ at $\mathfrak{l}$ are type $I_4^*$ and type $II^*$, while in case (ii), the possible reduction types are type $II$ and type $I_8^*$. In all these cases, the local Tamagawa number $c_{\mathfrak{l}}(E_{\delta}/F)$ is again is equal to $1,2$ or $4$. 

\section{Completion of proofs}

It remains to establish Lemma 3.4. The arguments are independent of the previous sections. We first begin with a more general setup as follows. 

\bigskip

Let $F$ be a totally real field, and $\mathfrak{n}$ an ideal of $\mathcal{O}_F$. An admissible factorization $\mathfrak{n} = \mathfrak{n}_{+} \mathfrak{n}_{-}$ is defined by the condition that $\mathfrak{n}_+, \mathfrak{n}_-$ is a pair of relatively prime ideals of $\mathcal{O}_F$, and such that the following holds for $\mathfrak{n}_-$: in the case where $[F:\mathbf{Q}]$ is odd, we assume that $\mathfrak{n}_{-}$ is a square-free product of an even number of distinct prime ideals of $\mathcal{O}_F$ (we allow $\mathfrak{n}_{-}=\mathcal{O}_F$), while in the case where $[F:\mathbf{Q}]$ is even, we assume that $\mathfrak{n}_{-}$ is a square-free product of an odd number of distinct prime ideals of $\mathcal{O}_F$. Given an admissible factorization $\mathfrak{n} = \mathfrak{n}_+ \mathfrak{n}_-$, note that if $\mathfrak{p}, \mathfrak{q}$ are distinct prime ideals that divide $\mathfrak{n}_-$, then the pair $\mathfrak{n}_+ \mathfrak{p} \mathfrak{q}$ and $\mathfrak{n}_-  (\mathfrak{p} \mathfrak{q})^{-1}$ again gives an admissible factorization of $\mathfrak{n}$. 

\bigskip

Fix an archimedean place $\nu$ of $F$ corresponding to field embedding $\iota: F \hookrightarrow \mathbf{R}$. We denote by $X(\mathfrak{n}_{+},\mathfrak{n}_{-})/F$ the (compact) Shimura curve over $F$ of Eichler level $\mathfrak{n}_+$, defined with respect to the quaternion algebra over $F$ that ramifies exactly at the prime ideals of $\mathcal{O}_F$ dividing $\mathfrak{n}_-$, and at the archimedean places of $F$ other than $\nu$. In the special case where $F=\mathbf{Q}$, ideals $\mathfrak{n}_+,\mathfrak{n}_-$ corresponds to positive integers $N_+,N_-$, and $\mathfrak{n}$ corresponds to $N:=N_+ N_-$. In the case where $N_- =1$, so $N=N_+$, then $X(N,1)$ is taken to be the (compactified) modular curve $X_0(N)/\mathbf{Q}$. 

\bigskip

Consider an elliptic curve $E/F$ with conductor equal to $\mathfrak{n}$. We assume that $E/F$ has multiplicative reduction at some prime dividing $\mathfrak{n}$ (thus this is automatic when $\mathfrak{n}_- \neq \mathcal{O}_F$, for instance when $[F:\mathbf{Q}]$ is even). In particular $E/\overline{F}$ does not have complex multiplication. Also assume that $E/F$ is modular, namely that $E/F$ is associated to a parallel weight two cuspidal Hilbert eigenform $\mathbf{f}$ over $F$ of level $\mathfrak{n}$ (thus in particular the Hecke eigenvalues of $\mathbf{f}$ all belong to $\mathbf{Z}$).

\bigskip

By the Jacquet-Langlands correspondence and the isogeny theorem, one has that $E/F$ admits a modular parametrization over $F$ by the Shimura curve $X(\mathfrak{n}_+,\mathfrak{n}_-)/F$, for any given admissible factorization $\mathfrak{n} = \mathfrak{n}_+ \mathfrak{n}_-$. More precisely, denote by $J(\mathfrak{n}_+,\mathfrak{n}_-)/F$ the Jacobian variety of $X(\mathfrak{n}_+,\mathfrak{n}_-)/F$ (see for example p. 29 of [Z1] or section 5.1 of [H] for the definition of the Jacobian variety of $X(\mathfrak{n}_+,\mathfrak{n}_-)/F$). Then there exists a non-constant morphism of abelian varieties over $F$:
\[
h : J(\mathfrak{n}_+,\mathfrak{n}_-)/F \rightarrow E/F
\]
which will be referred to as a modular parametrization of $E/F$ by $X(\mathfrak{n}_+,\mathfrak{n}_-)/F$ or by $J(\mathfrak{n}_+,\mathfrak{n}_-)/F$. Corresponding to $h$, denote by:
\[
h^{\vee} : = E/F \rightarrow J(\mathfrak{n}_+,\mathfrak{n}_-)/F
\]
the dual morphism (here we are using the fact that $E/F$ and $J(\mathfrak{n}_+,\mathfrak{n}_-)$ are canonically self-dual). The degree $\deg(h)$ of the modular parametrization $h$, is defined to be the unique positive integer $d$, such that
\[
[d] = h \circ h^{\vee}
\]
with $[d]$ being the multiplication by $d$ endomorphism on $E/F$. The choice of $h$ is not unique, but since $E/\overline{F}$ does not have complex multiplication, one has that, the degrees of the different choices of modular parametrizations of $E/F$ by $X(\mathfrak{n}_+,\mathfrak{n}_-)/F$ differ by multiplication by squares of non-zero rational numbers.

\bigskip

A modular parametrization $h: J(\mathfrak{n}_+,\mathfrak{n}_-)/F \rightarrow E/F$ as above is said to be {\it optimal}, if $ker(h)$ is connected, i.e. an abelian sub-variety of $J(\mathfrak{n}_+,\mathfrak{n}_-)/F$, in which case $E/F$ is said to be an optimal quotient of $J(\mathfrak{n}_+,\mathfrak{n}_-)/F$ (via $h$). In general given a modular parametrization $h:J(\mathfrak{n}_+,\mathfrak{n}_-)/F \rightarrow E/F$, there is an elliptic curve $E^{\prime}/F$ that is isogeneous over $F$ to $E/F$, and such that $E^{\prime}/F$ admits an optimal modular parametrization $h^{\prime}: J(\mathfrak{n}_+,\mathfrak{n}_-)/F \rightarrow E^{\prime}/F$ (in particular $E^{\prime}/F$ corresponds to the same cuspidal Hilbert eigenform $\mathbf{f}$).

\bigskip

Recall that we assumed that $E/F$ has multiplicative reduction at some prime dividing $\mathfrak{n}$. Let $\mathfrak{p}$ be one such prime of multiplicative reduction (thus $\mathfrak{p}$ exactly divides $\mathfrak{n}$). Put $\overline{c}_{\mathfrak{p}}(E/F) := \ord_{\mathfrak{p}} \Delta^{min}(E/F)$. We first observe the following: if $g: E/F \rightarrow E^{\prime}/F$ is an isogeny over $F$ (so $E^{\prime}/F$ again has multiplicative reduction at $\mathfrak{p}$), then 
\[
\deg(g) \cdot \overline{c}_{\mathfrak{p}}(E^{\prime}/F) = \overline{c}_{\mathfrak{p}}(E/F)  \mod{(\mathbf{Q}^{\times})^{2}}.
\]
This follows from Tate's uniformization theory, which in particular describes the set of all isogenies over $F_{\mathfrak{p}^2}$ between $E/F_{\mathfrak{p}^2}$ and $E^{\prime}/F_{\mathfrak{p}^2}$ (here $F_{\mathfrak{p}^2}$ is the quadratic unramified extension of the completion $F_{\mathfrak{p}}$ of $F$ at $\mathfrak{p}$), in terms of the Tate $q$-parameters $q(E/F_{\mathfrak{p}}), q(E^{\prime}/F_{\mathfrak{p}}) \in F_{\mathfrak{p}}^{\times}$, {\it c.f.} Theorem on p. 177 of [T2], and the fact that $\overline{c}_{\mathfrak{p}}(E/F) = \ord_{\mathfrak{p}} q(E/F_{\mathfrak{p}}), \overline{c}_{\mathfrak{p}}(E^{\prime}/F) = \ord_{\mathfrak{p}} q(E^{\prime}/F_{\mathfrak{p}})$ (thus the above relation in fact holds for isogeny over $F_{\mathfrak{p}^2}$ between $E/F_{\mathfrak{p}^2}$ and $E^{\prime}/F_{\mathfrak{p}^2}$).

\bigskip

From this it follows that, if $E^{\prime}/F$ is isogeneous over $F$ to $E/F$, then for any modular parametrization $h^{\prime} :J(\mathfrak{n}_+,\mathfrak{n}_-)  \rightarrow E^{\prime}/F$, one has:
\begin{eqnarray}
 \deg(h^{\prime}) \cdot \overline{c}_{\mathfrak{p}}(E^{\prime}/F) = \deg(h) \cdot \overline{c}_{\mathfrak{p}}(E/F)  \mod{(\mathbf{Q}^{\times})^{2}}.
\end{eqnarray}
This will be used repeatedly below.

\begin{proposition}
Let $\mathfrak{p},\mathfrak{q}$ be distinct prime ideals that divide $\mathfrak{n}_-$. For any modular parametrizations:
\[
h: J( \mathfrak{n}_+,\mathfrak{n}_-)/F \rightarrow E/F
\]
\[
\widetilde{h} : J(\mathfrak{n}_+ \mathfrak{p} \mathfrak{q}, \mathfrak{n}_- (\mathfrak{p} \mathfrak{q})^{-1})/F \rightarrow E/F
\]
we have the relation:
\[
\deg(\widetilde{h}) = \overline{c}_{\mathfrak{p}}(E/F) \cdot \overline{c}_{\mathfrak{q}}(E/F) \cdot \deg(h)   \mod{(\mathbf{Q}^{\times})^{2}}
\]
\end{proposition}
\begin{proof}
Consider optimal parametrizations:
\[
h^{\prime}: J( \mathfrak{n}_+,\mathfrak{n}_-)/F \rightarrow E^{\prime}/F
\]
\[
\widetilde{h}^{\prime \prime} : J(\mathfrak{n}_+ \mathfrak{p} \mathfrak{q}, \mathfrak{n}_- (\mathfrak{p} \mathfrak{q})^{-1})/F \rightarrow E^{\prime \prime} /F
\]
with $E^{\prime}/F$ and $E^{\prime \prime}/F$ both being isogeneous over $F$ to $E/F$. Using (4.1), we see that it suffices to show that
\[
\deg(\widetilde{h}^{\prime \prime}) = \overline{c}_{\mathfrak{p}}(E^{\prime} /F) \cdot \overline{c}_{\mathfrak{q}}(E^{\prime \prime}/F) \cdot \deg(h^{\prime})   \mod{(\mathbf{Q}^{\times})^{2}}.
\]
In the case where $F=\mathbf{Q}$, this follows from the result proved by Ribet-Takahashi, see Theorem 2 of [RT]. For general totally real $F$, it is given in Theorem 3.2.7 of [De]. 
\end{proof}

\begin{corollary}
As before consider an admissible factorization $\mathfrak{n} =\mathfrak{n}_+ \mathfrak{n}_-$. In the case where $[F:\mathbf{Q}]$ is odd, consider two modular parametrizations of $E/F$:
\[
h: J(\mathfrak{n}_+,\mathfrak{n}_-)/F \rightarrow E/F
\]
\[
\widetilde{h} : J(\mathfrak{n},\mathcal{O}_F)/F \rightarrow E/F
\]
then we have
\[
\deg(\widetilde{h}) = \big( \prod_{\mathfrak{q} | \mathfrak{n}_-} \overline{c}_{\mathfrak{q}}(E/F) \big) \cdot  \deg(h)   \mod{(\mathbf{Q}^{\times})^{2}}.
\]
In the case where $[F:\mathbf{Q}]$ is even, let $\mathfrak{p}$ be a prime ideal dividing $\mathfrak{n}_-$. Consider two modular parametrizations of $E/F$:
\[
h: J(\mathfrak{n}_+,\mathfrak{n}_-)/F \rightarrow E/F
\]
\[
\widetilde{h} : J(\mathfrak{n} \mathfrak{p}^{-1},\mathfrak{p})/F \rightarrow E/F
\]
then we have
\[
\deg(\widetilde{h}) = \big( \prod_{\mathfrak{q} | \mathfrak{n}_- \mathfrak{p}^{-1} } \overline{c}_{\mathfrak{q}}(E/F) \big) \cdot  \deg(h)   \mod{(\mathbf{Q}^{\times})^{2}}.
\]
\end{corollary}
\begin{proof}
This follows from successively applying Proposition 4.1, on recalling that in the case where $[F:\mathbf{Q}]$ is odd, that $\mathfrak{n}_-$ is a square-free product of an even number of distinct prime ideals of $\mathcal{O}_F$, while in the case where $[F:\mathbf{Q}]$ is even, that $\mathfrak{n}_- \mathfrak{p}^{-1}$ is a square-free product of an even number of distinct prime ideals of $\mathcal{O}_F$.
\end{proof}

\bigskip

To continue we first define some more quantities. 

\bigskip
As before let $\mathfrak{n} = \mathfrak{n}_+ \mathfrak{n}_-$ be an admissible factorization, and let $\mathfrak{p}$ be a prime that exactly divides $\mathfrak{n}$. The Shimura curve $X(\mathfrak{n}_+,\mathfrak{n}_-)/F$ has semi-stable reduction at the prime $\mathfrak{p}$. In the case where $\mathfrak{p}$ divides $\mathfrak{n}_+$, and $F=\mathbf{Q}$, this is due to Deligne-Rapoport [DR] ($N_-=1$) and K. Buzzard [Bu] ($N_- \neq 1$), while for $F \neq \mathbf{Q}$ this is due to Carayol [Ca]. In the case where $\mathfrak{p}$ divides $\mathfrak{n}_-$, this is due to Cerednik [Ce] and Drinfeld [Dr] ({\it c.f.} also [Ku, BC, BZ]; for a summary of the results of Cerednik and Drinfeld, the reader may also refer to section 1.4 - 1.5 of [N]). These authors have constructed proper flat model $\mathcal{X}(\mathfrak{n}_+,\mathfrak{n}_-)/\mathcal{O}_{F_{\mathfrak{p}}}$ for $X(\mathfrak{n}_+,\mathfrak{n}_-)/F_{\mathfrak{p}}$, with semi-stable reduction (the model $\mathcal{X}(\mathfrak{n}_+,\mathfrak{n}_-)/\mathcal{O}_{F_{\mathfrak{p}}}$ is admissible in the sense of Jordan-Livn\'e [JL]). 

\bigskip
The Jacobian $J(\mathfrak{n}_+,\mathfrak{n}_-)/F$ thus also has semi-stable reduction at the prime $\mathfrak{p}$. More precisely, denote by $\mathcal{J}(\mathfrak{n}_+,\mathfrak{n}_-)/\mathcal{O}_{F_{\mathfrak{p}}}$ the Neron model over $\mathcal{O}_{F_{\mathfrak{p}}}$ of $J(\mathfrak{n}_+,\mathfrak{n}_-)/F_{\mathfrak{p}}$. Consider the special fibre $\mathcal{J}(\mathfrak{n}_+,\mathfrak{n}_-)_{k_{\mathfrak{p}}}$ over the reside field $k_{\mathfrak{p}}$ of $\mathcal{O}_{F_{\mathfrak{p}}}$. By work of Raynaud [Ra] (see for example section 2 of [Ri] for a summary), we have, as a consequence of the integral semi-stable models $\mathcal{X}(\mathfrak{n}_+,\mathfrak{n}_-)/\mathcal{O}_{F_{\mathfrak{p}}}$ for $X(\mathfrak{n}_+,\mathfrak{n}_-)/F_{\mathfrak{p}}$ as mentioned above, the following: the identity component $\mathcal{J}(\mathfrak{n}_+,\mathfrak{n}_-)^0_{k_{\mathfrak{p}}}$ is a semi-abelian variety over $k_{\mathfrak{p}}$, and we have a short exact sequence:
\[
1 \rightarrow  \mathcal{T} \rightarrow \mathcal{J}(\mathfrak{n}_+,\mathfrak{n}_-)^0_{k_{\mathfrak{p}}} \rightarrow   \mathcal{A}    \rightarrow  1
\]
where $\mathcal{T}$ is a torus over $k_{\mathfrak{p}}$ and $\mathcal{A}$ is an abelian variety over $k_{\mathfrak{p}}$ (in the case where $\mathfrak{p}$ divides $\mathfrak{n}_-$ then in fact $\mathcal{A}$ is trivial, as follows from the integral model $\mathcal{X}(\mathfrak{n}_+,\mathfrak{n}_-)/\mathcal{O}_{F_{\mathfrak{p}}}$ given by Cerednik and Drinfeld). Put
\[
\mathfrak{X}_{\mathfrak{p}} = \mathfrak{X}_{\mathfrak{p}}(\mathfrak{n}_+,\mathfrak{n}_-) := \Hom(\mathcal{T}_{/\overline{k}_{\mathfrak{p}}},\mathbf{G}_{m / \overline{k}_{\mathfrak{p}}})
\]
which is a finite free $\mathbf{Z}$-module equipped with action of $\Gal(\overline{k}_{\mathfrak{p}}/k_{\mathfrak{p}})$ and Hecke operators at primes not dividing $\mathfrak{n}$.

\bigskip

By Grothendieck [G, Theorem 10.4], there is a monodromy pairing $(  \cdot , \cdot   )_{\mathfrak{p}} : \mathfrak{X}_{\mathfrak{p}}(\mathfrak{n}_+,\mathfrak{n}_-)\times \mathfrak{X}_{\mathfrak{p}}(\mathfrak{n}_+,\mathfrak{n}_-) \rightarrow \mathbf{Z}$ (we are using again the auto-duality of Jacobian), which we refer to as the monodromy pairing with respect to $J(\mathfrak{n}_+,\mathfrak{n}_-)/F$ at the prime $\mathfrak{p}$. By Theorem 11.5 of [G], the monodromy pairing induces a short exact sequence:
\[
0 \rightarrow \mathfrak{X}_{\mathfrak{p}}(\mathfrak{n}_+,\mathfrak{n}_-) \rightarrow \Hom_{\mathbf{Z}}( \mathfrak{X}_{\mathfrak{p}}(\mathfrak{n}_+,\mathfrak{n}_-) ,\mathbf{Z}) \rightarrow  \Phi_{\mathfrak{p}}(\mathfrak{n}_+,\mathfrak{n}_-) \rightarrow 0
\]
where
\[
\Phi_{\mathfrak{p}}(\mathfrak{n}_+,\mathfrak{n}_-)  := \mathcal{J}(\mathfrak{n}_+,\mathfrak{n}_-)_{\overline{k}_{\mathfrak{p}}} / \mathcal{J}(\mathfrak{n}_+,\mathfrak{n}_-)^0_{\overline{k}_{\mathfrak{p}}}
\]
is the finite abelian group of the components of $\mathcal{J}(\mathfrak{n}_+,\mathfrak{n}_-)_{\overline{k}_{\mathfrak{p}}}$.

\bigskip

Now as before let $h: J(\mathfrak{n}_+,\mathfrak{n}_-)/F \rightarrow E/F$ be a modular parametrization. Then $E/F$ has multiplicative reduction at the prime $\mathfrak{p}$. Denote by $\mathcal{E}/\mathcal{O}_{F_{\mathfrak{p}}}$ the Neron model of $E/F_{\mathfrak{p}}$ over $\mathcal{O}_{F_{\mathfrak{p}}}$. Put:
\[
\Phi_{\mathfrak{p}}(E/F) : = \mathcal{E}_{\overline{k}_{\mathfrak{p}}} /\mathcal{E}^0_{\overline{k}_{\mathfrak{p}}}
\]
then $\# \Phi_{\mathfrak{p}}(E/F) = \overline{c}_{\mathfrak{p}}(E/F)$. Denote by:
\[
h_{*} : \Phi_{\mathfrak{p}}(\mathfrak{n}_+,\mathfrak{n}_-)  \rightarrow \Phi_{\mathfrak{p}}(E/F)
\]
the map on the component groups induced by $h$. 

\bigskip

The $\mathbf{Z}$-module $\mathfrak{X}_{\mathfrak{p}}(\mathbf{n}_+,\mathfrak{n}_-)$ satisfies a multiplicity one property with respect to the action of the Hecke operators at primes not dividing $\mathfrak{n}$, namely that the $\mathbf{f}$-eigen-submodule of $\mathfrak{X}_{\mathfrak{p}}(\mathfrak{n}_+,\mathfrak{n}_-)$ (with respect to the action of the Hecke operators at primes not dividing $\mathfrak{n}$) is of $\mathbf{Z}$-rank one; in the case $F=\mathbf{Q}$ {\it c.f.} for example the argument in the proof of Proposition 1 in [RT]; the general totally real case is similar. Denote by $\xi_{\mathbf{f}}$ a $\mathbf{Z}$-generator.

\begin{proposition}
Let $\mathfrak{n} = \mathfrak{n}_+ \mathfrak{n}_-$ be an admissible factorization. Suppose that
\[
h: J(\mathfrak{n}_+,\mathfrak{n}_-) /F \rightarrow E/F
\]
is an optimal parametrization, and that $\mathfrak{p}$ is a prime ideal that exactly divides $\mathfrak{n}$. Then with notations as above, we have:
\[
\deg(h) =  \frac{ \overline{c}_{\mathfrak{p}}(E/F) }{(\#  im (h_{*}))^2}  \cdot (\xi_{\mathbf{f}} , \xi_{\mathbf{f}} )_{\mathfrak{p}}    =  \frac{  \# coker(h_{*}) }{\#  im (h_{*})}  \cdot (\xi_{\mathbf{f}} , \xi_{\mathbf{f}} )_{\mathfrak{p}}.
\]
\end{proposition}
\begin{proof}
In the case where $F=\mathbf{Q}$ this is due to Ribet-Takahashi, stated for example as Theorem 2.3 in [Tak]. The argument for a general totally real $F$ is essentially the same, and is given for instance in Theorem 3.2.6 in [De]. 
\end{proof}

\begin{corollary}
Let $\mathfrak{n} = \mathfrak{n}_+ \mathfrak{n}_-$ be an admissible factorization, and $\mathfrak{p}$ be a prime ideal that exactly divides $\mathfrak{n}$. For any modular parametrization:
\[
h: J(\mathfrak{n}_+,\mathfrak{n}_-) /F \rightarrow E/F
\]
we have:
\[
\deg(h) =  \overline{c}_{\mathfrak{p}}(E/F) \cdot  ( \xi_{\mathbf{f}} , \xi_{\mathbf{f}} )_{\mathfrak{p}}    \mod{(\mathbf{Q}^{\times})^{2}}.
\]
\end{corollary}
\begin{proof}
In the case where $h$ is an optimal parametrization, this follows from Proposition 4.3. In general, there is an $E^{\prime}/F$ that is isogeneous over $F$ to $E/F$, with $E^{\prime}/F$ admitting an optimal parametrization $h^{\prime} : J(\mathfrak{n}_+,\mathfrak{n}_-) \rightarrow E^{\prime}/F$. Then by applying Proposition 4.3 to $h^{\prime}$, and then using (4.1), we obtain Corollary 4.4 in general. 
\end{proof}

\begin{corollary}
Let $\mathfrak{n} =\mathfrak{n}_+ \mathfrak{n}_-$ be an admissible factorization, and $\mathfrak{p}$ be a prime ideal that exactly divides $\mathfrak{n}$. 

\bigskip
If $[F:\mathbf{Q}]$ is odd and $\mathfrak{p}$ divides $\mathfrak{n}_+$, then given a modular parametrization:
\[
h: J(\mathfrak{n},\mathcal{O}_F) /F \rightarrow E/F
\]
we have the following:
\[
\deg(h) =  \big( \prod_{\mathfrak{q} | \mathfrak{n}_- \mathfrak{p}} \overline{c}_{\mathfrak{q}}(E/F)  \big) \cdot  ( \xi_{\mathbf{f}} , \xi_{\mathbf{f}} )_{\mathfrak{p}}    \mod{(\mathbf{Q}^{\times})^{2}}
\]
where the monodromy pairing is with respect to $J(\mathfrak{n}_+,\mathfrak{n}_-)/F$ at the prime $\mathfrak{p}$.

\bigskip
If $[F:\mathbf{Q}]$ is even and $\mathfrak{p}$ divides $\mathfrak{n}_-$, then given a modular parametrization:
\[
h: J(\mathfrak{n}_+,\mathfrak{n}_-)/F \rightarrow E/F
\]
we have the following:
\[
\deg(h) =  \big( \prod_{\mathfrak{q} | \mathfrak{n}_-} \overline{c}_{\mathfrak{q}}(E/F)  \big) \cdot  ( \xi_{\mathbf{f}} , \xi_{\mathbf{f}} )_{\mathfrak{p}}    \mod{(\mathbf{Q}^{\times})^{2}}
\]
where the monodromy pairing is with respect to $J(\mathfrak{n} \mathfrak{p}^{-1}, \mathfrak{p})/F$ at the prime $\mathfrak{p}$.
\end{corollary}
\begin{proof}
This follows from Corollary 4.2 and Corollary 4.4.
\end{proof}

\bigskip
We can now begin the proof of Lemma 3.4. We first show (3.18). Thus for the moment $F=\mathbf{Q}$. As in the previous sections $E/\mathbf{Q}$ is an elliptic curve corresponding to weight two cuspidal eigenform $f=f_E$ of level $N$, with $N$ being the conductor of $E/\mathbf{Q}$, such that $N= M \mathcal{Q}$, where $M$ and $\mathcal{Q}$ are relatively prime, and $\mathcal{Q}$ is a square-free product of an odd number of distinct primes. Pick a prime $p$ that divides $\mathcal{Q}$. Then put $N_+ :=M p$ and $N_- :=\mathcal{Q}/p$. The factorization $N=N_+ N_-$ is then an admissible factorization.  

\bigskip

We now apply Corollary 4.5 (for the case $F=\mathbf{Q}$ so in particular $[F:\mathbf{Q}]=1$ is odd), with $N_+=M p$ and $N_-=\mathcal{Q}/p$ (note that $p$ exactly divides $N_+$). Then for a modular parametrization of $E/\mathbf{Q}$ by $X(N,1)/\mathbf{Q} =X_0(N)/\mathbf{Q}$, or equivalently by the Jacobian $J_0(N)/\mathbf{Q}$:
\[
h: J_0(N)/\mathbf{Q} \rightarrow E/\mathbf{Q}
\]
we have
\begin{eqnarray}
\deg(h) = \big( \prod_{q | \mathcal{Q}} \overline{c}_q(E/\mathbf{Q}) \big) \cdot (\xi_{f_E},\xi_{f_E})_p   \mod{(\mathbf{Q}^{\times})^{2}}.
\end{eqnarray}
where the monodromy pairing is with respect to $J(M p,\mathcal{Q}/p)/\mathbf{Q}$ at the prime $p$.

\bigskip

We now explicate the monodromy pairing. Firstly, as a consequence of the integral model of Deligne-Rapoport [DR] and K. Buzzard [Bu], the work of Raynaud [Ra] gives a canonical identification between $\mathfrak{X}_p(M p,\mathcal{Q}/p)$ and the $\mathbf{Z}$-module of degree zero divisors on the supersingular points on $\mathcal{X}(M p,\mathcal{Q}/p)_{\overline{\mathbf{F}}_p}$ (namely the non-smooth points on $\mathcal{X}(M p,\mathcal{Q}/p)_{\overline{\mathbf{F}}_p}$), and the set of supersingular points on $\mathcal{X}(M p,\mathcal{Q}/p)_{\overline{\mathbf{F}}_p}$ can be identified with the finite automorphic quotient:
\[
B^{\times} \backslash \widehat{B}^{\times} / \widehat{R}^{\times}
\]
where $B$ is the definite quaternion algebra over $\mathbf{Q}$ that ramifies exactly at primes dividing $\mathcal{Q}$ and the archimedean place, and $R$ is an Eichler order of $B$ of level $M$, and this identification respect the action of the Hecke operators at primes not dividing $N=M \mathcal{Q}$. See for example Proposition 3.1 and 3.3 of [Ri]; here in the case when $\mathcal{Q}/p \neq 1$, the arguments in [Ri] apply verbatim to $\mathcal{X}(M p ,\mathcal{Q}/p)_{\mathbf{F}_p}$, when the results of [Bu] is substituted in place of [DR] in the discussion in {\it loc. cit.}
\bigskip

Fix choices of representatives for elements $ b \in B^{\times} \backslash \widehat{B}^{\times} / \widehat{R}^{\times}$ in $\widehat{B}^{\times}$, and we still denote the representative as $b$. Put:
\[
R_b = B \cap b \widehat{R} b^{-1}
\]
then $R_b$ is again Eichler order of $B$ of level $M$. Define:
\[
e_b : = \# (R_b^{\times}/ \{\pm1 \}).
\]
(thus $e_b$ is independent of the choice of representative of $b  \in B^{\times} \backslash \widehat{B}^{\times} / \widehat{R}^{\times}$ in $\widehat{B}^{\times}$). Then under the previous identification of $\mathfrak{X}_p(Mp,\mathcal{Q}/p)$ with the $\mathbf{Z}$-module of degree zero divisors on the finite set $B^{\times} \backslash \widehat{B}^{\times} / \widehat{R}^{\times}$, the monodromy pairing $( \cdot , \cdot )_p$ on $\mathfrak{X}_p(Mp, \mathcal{Q}/p)$ is given by (the restriction of) the $\mathbf{Z}$-bilinear symmetric paring on the $\mathbf{Z}$-module of divisors on $B^{\times} \backslash \widehat{B}^{\times} / \widehat{R}^{\times}$, determined by the condition: $([b],[b^{\prime}]    )_p = e_b$ if $b=b^{\prime}$, and is equal to $0$ otherwise, {\it c.f.} again the discussions in section 2 and p. 448 of section 3 in [Ri].  

\bigskip

Now let $\phi_E$ be scalar-valued automorphic eigenform (with trivial central character) with respect to $B^{\times}$ that corresponds to $f_E$ under the Jacquet-Langlands correspondence. Thus $\phi_E$ is a function on $B^{\times} \backslash \widehat{B}^{\times} / \widehat{R}^{\times}$, and we chose $\phi_E$ so that it takes values in $\mathbf{Q}$. Recall that the Petersson inner product on $\mathbf{C}$-valued functions on $B^{\times} \backslash \widehat{B}^{\times} / \widehat{R}^{\times}$ (with trivial central character) is defined by the following: for $\phi_1,\phi_2$ two such functions,
\[
\langle \phi_1,\phi_2 \rangle : = \sum_{b \in B^{\times} \backslash \widehat{B}^{\times} / \widehat{R}^{\times}} \frac{1}{e_b} \phi_1(b) \cdot  \overline{\phi_2(b)}.
\]
The action of the Hecke operators at primes not dividing $N$ is self-adjoint with respect to the Petersson inner product.

\bigskip

Since $\phi_E$ corresponds to the cusp form $f_E$, we have that $\phi_E$ is orthogonal to the constant functions on $B^{\times} \backslash \widehat{B}^{\times} / \widehat{R}^{\times}$. In particular:
\[
0 = \langle \phi_E,1 \rangle = \sum_{b \in B^{\times} \backslash \widehat{B}^{\times} / \widehat{R}^{\times}} \frac{1}{e_b} \phi_E(b)
\]

Now let $t$ be a non-zero integer such that 
\[
t \cdot \frac{1}{e_b} \phi_E(b) \in \mathbf{Z}
\]
for all $b \in B^{\times} \backslash \widehat{B}^{\times} / \widehat{R}^{\times}$. Then 
\[
\zeta_{\phi_E} := \sum_{b \in B^{\times} \backslash \widehat{B}^{\times} / \widehat{R}^{\times}}  ( t \cdot \frac{1}{e_b} \phi_E(b)) \cdot [b]
\]
is a degree zero divisor on $B^{\times} \backslash \widehat{B}^{\times} / \widehat{R}^{\times}$, hence a non-zero element of $\mathfrak{X}_p(Mp, \mathcal{Q}/p)$. Note that $(\zeta_{\phi_E},\zeta_{\phi_E})_p = t^2 \langle \phi_E,\phi_E \rangle$. In addition since $\phi_E$ corresponds to $f_E$ under the Jacquet-Langlands correspondence, one has that $\zeta_{\phi_E}$ belongs to the $f_E$-eigen-submodule of $\mathfrak{X}_p(Mp, \mathcal{Q}/p)$, and hence 
\[
\zeta_{\phi_E} = s \cdot \xi_{f_E}
\] 
for some non-zero integer $s$. In particular $(\zeta_{\phi_E},\zeta_{\phi_E})_p=s^2 (\xi_{f_E},\xi_{f_E})_p$.

\bigskip

Summarizing, we have 
\[
(\xi_{f_E},\xi_{f_E})_p = \langle \phi_E , \phi_E \rangle  \mod{(\mathbf{Q}^{\times})^{2}}.
\]
and so (3.18) follows on combining this with (4.2).

\bigskip

To complete the proof of Lemma 3.4, it remains to establish (3.19). For the moment we again consider the more general case. Thus $F$ is a totally real field, but with $[F:\mathbf{Q}]$ even, and $\mathfrak{n}$ an ideal of $\mathcal{O}_F$ with admissible factorization $\mathfrak{n} = \mathfrak{n}_+ \mathfrak{n}_-$ (recall that in this case, $\mathfrak{n}_-$ is a square-free product of an odd number distinct prime ideals of $\mathcal{O}_F$). 

\bigskip

Let $E/F$ be an elliptic curve of conductor $\mathfrak{n}$ that is associated to parallel weight two cuspidal Hilbert eigenform $\mathbf{f}=\mathbf{f}_E$ over $F$ of level $\mathfrak{n}$, and 
\[
h: J(\mathfrak{n}_+,\mathfrak{n}_- ) \rightarrow E/F
\]
be a modular parametrization. Denote by $B_F$ the quaternion algebra over $F$ that ramifies exactly at the archimedean places of $F$ (in particular $B_F$ splits at all the finite places of $F$), and let $R_F$ be an Eichler order of $B_F$ of level $\mathfrak{n}$. Now let $\Phi_E$ be scalar-valued automorphic eigenform (with trivial central character) with respect to $B_F^{\times}$ that corresponds to $\mathbf{f}_E$ under the Jacquet-Langlands correspondence. Thus $\Phi_E$ is a function on the finite automorphic quotient $B_F^{\times} \backslash \widehat{B}_F^{\times} / \widehat{R}_F^{\times}$. We choose $\Phi_E$ so that it takes values in $\mathbf{Q}$.

\bigskip

Now the Petersson inner product on $\mathbf{C}$-valued functions on $B_F^{\times} \backslash \widehat{B}_F^{\times} / \widehat{R}_F^{\times}$ (with trivial central character) is defined by the following: for $\Phi_1,\Phi_2$ two such functions,
\[
\langle \Phi_1,\Phi_2 \rangle : = \sum_{b \in B_F^{\times} \backslash \widehat{B}_F^{\times} / \widehat{R}_F^{\times}} \frac{1}{e_b} \Phi_1(b) \cdot  \overline{\Phi_2(b)}.
\]
Here $e_b$ is defined as follows. Fix choices of representatives for elements $ b \in B^{\times}_F \backslash \widehat{B}^{\times}_F / \widehat{R}^{\times}_F$ in $\widehat{B}^{\times}_F$, and we still denote the representative as $b$. Put:
\[
R_{F,b} = B_F \cap b \widehat{R}_F b^{-1}
\]
then $R_{F,b}$ is again Eichler order of $B_F$ of level $\mathfrak{n}$. We have:
\[
e_b : = \# (  R_{F,b}^{\times} / \mathcal{O}_F^{\times} )
\]
($e_b$ is independent of the choice of representation of $b \in B^{\times}_F \backslash \widehat{B}^{\times}_F / \widehat{R}^{\times}_F$ in $\widehat{B}^{\times}_F$). The action of the Hecke operators at primes not dividing $\mathfrak{n}$ is self-adjoint with respect to the Petersson inner product.

\bigskip

Now let $\mathfrak{p}$ be a prime dividing $\mathfrak{n}_-$. Then applying Corollary 4.5 to $h: J(\mathfrak{n}_+,\mathfrak{n}_- ) \rightarrow E/F$, we have:
\[
\deg(h) =  \big( \prod_{\mathfrak{q} | \mathfrak{n}_- } \overline{c}_{\mathfrak{q}}(E/F)  \big) \cdot  ( \xi_{\mathbf{f}_E} , \xi_{\mathbf{f}_E} )_{\mathfrak{p}}    \mod{(\mathbf{Q}^{\times})^{2}}
\]
where the monodromy pairing is with respect to $J(\mathfrak{n} \mathfrak{p}^{-1}, \mathfrak{p})/F$ at the prime $\mathfrak{p}$. We now explicate the monodromy pairing and relate it to the Petersson inner product $\langle \Phi_E,\Phi_E \rangle$.

\bigskip

As before $R_F$ is Eichler order of $B_F$ with level $\mathfrak{n}$; in addition let $\underline{R}_F$ be an Eichler order of $B_F$ of level $\mathfrak{n} \mathfrak{p}^{-1}$, such that $R_F \subset \underline{R}_F$. Now as a consequence of the integral models of Cerednik [Ce] and Drinfeld [Dr], the work of Raynaud [Ra] gives an injection of $\mathfrak{X}_{\mathfrak{p}}(\mathfrak{n} \mathfrak{p}^{-1},\mathfrak{p})$ into the $\mathbf{Z}$-module of degree zero divisors on the non-smooth points of $\mathcal{X}(\mathfrak{n} \mathfrak{p}^{-1},\mathfrak{p})_{\overline{k}_{\mathfrak{p}}}$, {\it c.f} section 2 of [Ri]. In addition, one has a description of the special fibre $\mathcal{X}(\mathfrak{n} \mathfrak{p}^{-1},\mathfrak{p})_{\overline{k}_{\mathfrak{p}}}$; equivalently a description of the the dual graph of $\mathcal{X}(\mathfrak{n} \mathfrak{p}^{-1},\mathfrak{p})_{\overline{k}_{\mathfrak{p}}}$, where the set of non-smooth points of $\mathcal{X}(\mathfrak{n} \mathfrak{p}^{-1},\mathfrak{p})_{\overline{k}_{\mathfrak{p}}}$ corresponds to the set of un-oriented edges, and the set of irreducible components of $\mathcal{X}(\mathfrak{n} \mathfrak{p}^{-1},\mathfrak{p})_{\overline{k}_{\mathfrak{p}}}$ (which are isomorphic to $\mathbf{P}^1_{\overline{k}_{\mathfrak{p}}}$) corresponds to the set of vertices. It is a bipartite graph. One has that the set of un-oriented edges of the dual graph can be identified with:
\[
B_F^{\times} \backslash \widehat{B}_F^{\times} / \widehat{R}_F^{\times} F^{\times}_{\mathfrak{p}}
\]
and this identification also gives an orientation to the dual graph. While the set of vertices of the dual graph can be identified with:
\[
(B_F^{\times} \backslash \widehat{B}_F^{\times} / \underline{\widehat{R}}_F^{\times} F^{\times}_{\mathfrak{p}}) \times \mathbf{Z}/2 \mathbf{Z}
\]
with the two copies of $B_F^{\times} \backslash \widehat{B}_F^{\times} / \underline{\widehat{R}}_F^{\times} F^{\times}_{\mathfrak{p}}$ in $(B_F^{\times} \backslash \widehat{B}_F^{\times} / \underline{\widehat{R}}_F^{\times} F^{\times}_{\mathfrak{p}}) \times \mathbf{Z}/2 \mathbf{Z}$ forming the partition of the set of vertices of the dual graph into two disjoint subsets of vertices, and which gives the dual graph the structure of a bipartite graph that respects the orientation. These identifications respect the action of Hecke operators at primes not dividing $\mathfrak{n}$

\bigskip
Under these identifications, the image of $\mathfrak{X}_{\mathfrak{p}}(\mathfrak{n} \mathfrak{p}^{-1},\mathfrak{p})$ in the $\mathbf{Z}$-module of degree zero divisors on $B_F^{\times} \backslash \widehat{B}_F^{\times} / \widehat{R}_F^{\times} F^{\times}_{\mathfrak{p}}$ is equal to the kernel of the degeneracy maps from level $\mathfrak{n}$ to level $\mathfrak{n} \mathfrak{p}^{-1}$, namely equal to the kernel of the degeneracy maps from the $\mathbf{Z}$-module of degree zero divisors on $B_F^{\times} \backslash \widehat{B}_F^{\times} / \widehat{R}_F^{\times} F^{\times}_{\mathfrak{p}}$, to the $\mathbf{Z}$-module of (degree zero) divisors on $B_F^{\times} \backslash \widehat{B}_F^{\times} /\underline{ \widehat{R}}_F^{\times} F^{\times}_{\mathfrak{p}}$; in terms of the dual graph of $\mathcal{X}(\mathfrak{n} \mathfrak{p}^{-1},\mathfrak{p})_{\overline{k}_{\mathfrak{p}}}$ and the choice of orientation as above, the degeneracy maps correspond to the two maps: sending an oriented edge to its source, and sending an oriented edge to its target; {\it c.f.} section 2 of [Ri], section 1.4 - 1.5 of [N]. 

\bigskip

Fix choices of representatives for elements $ c \in B^{\times}_F \backslash \widehat{B}^{\times}_F / \widehat{R}^{\times}_F F^{\times}_{\mathfrak{p}}$ in $\widehat{B}^{\times}_F$, and we still denote the representative as $c$. Similar to before put:
\[
R_{F,c} = B_F \cap c \widehat{R}_F c^{-1}
\]
then $R_{F,c}$ is again Eichler order of $B_F$ of level $\mathfrak{n}$. Define:
\[
e^{(\mathfrak{p})}_c : = \# ((R_{F,c}^{(\mathfrak{p})})_+^{\times}/   (\mathcal{O}_F^{(\mathfrak{p})})^{\times} ).
\]
Here in the definition of $e_c^{(\mathfrak{p})}$, the terms on the right are defined as follows. As for $\mathcal{O}_F^{(\mathfrak{p})}$ it is the subring of $F$ consisting of elements that are integral at all finite places of $F$ other than $\mathfrak{p}$. Put $R_{F,c}^{(\mathfrak{p})} = R_{F,c} \otimes_{\mathcal{O}_F} \mathcal{O}_F^{(\mathfrak{p})}$, and $(R_{F,c}^{(\mathfrak{p})})_+^{\times}$ is the group of elements $\gamma \in (R_{F,c}^{(\mathfrak{p})})^{\times}$ such that the reduced norm $N_{B_F/F}(\gamma) \in ( \mathcal{O}_F^{(\mathfrak{p})})^{\times}$ satisfies $\ord_{\mathfrak{p}}N_{B_F/F}(\gamma) \equiv 0 \mod{2}$ ($e_c^{(\mathfrak{p})}$ is independent of the choice of representative of $c \in B_F^{\times} \backslash \widehat{B}_F^{\times} / \widehat{R}_F^{\times} F^{\times}_{\mathfrak{p}}$ in $\widehat{B}_F^{\times}$).

\bigskip

Then under the previous identification of $\mathfrak{X}_{\mathfrak{p}}(\mathfrak{n} \mathfrak{p}^{-1},\mathfrak{p})$ with a $\mathbf{Z}$-submodule of the $\mathbf{Z}$-module of degree zero divisors on the finite set $B_F^{\times} \backslash \widehat{B}_F^{\times} / \widehat{R}_F^{\times} F^{\times}_{\mathfrak{p}}$, the monodromy pairing $( \cdot , \cdot )_{\mathfrak{p}}$ on $\mathfrak{X}_{\mathfrak{p}}(\mathfrak{n} \mathfrak{p}^{-1},\mathfrak{p})$ is given by (the restriction of) the $\mathbf{Z}$-bilinear symmetric paring on the $\mathbf{Z}$-module of divisors on $B_F^{\times} \backslash \widehat{B}_F^{\times} / \widehat{R}_F^{\times} F^{\times}_{\mathfrak{p}}$, determined by the condition: $([c],[c^{\prime}]    )_{\mathfrak{p}} = e^{(\mathfrak{p})}_c$ if $c=c^{\prime}$, and is equal to $0$ otherwise, {\it c.f.} again section 2 of [Ri], and Proposition 3.2 of [Ku], section 1.5.3 of [N]. 

\bigskip

To compare the monodromy pairing with the Petersson inner product, we assume in addition that $\mathfrak{p}$ is inert over $\mathbf{Q}$; thus $\mathfrak{p} = p \mathcal{O}_F$ where $p$ is a prime number. Then with this $\mathfrak{p}$, we have 
\[
B_F^{\times} \backslash \widehat{B}_F^{\times} / \widehat{R}_F^{\times} = B_F^{\times} \backslash \widehat{B}_F^{\times} / \widehat{R}_F^{\times} F^{\times}_{\mathfrak{p}}
\]
(and so the set of representatives of $B_F^{\times} \backslash \widehat{B}_F^{\times} / \widehat{R}_F^{\times} $ and $B_F^{\times} \backslash \widehat{B}_F^{\times} / \widehat{R}_F^{\times} F^{\times}_{\mathfrak{p}}$ in $\widehat{B}_F^{\times}$ would be taken to be the same set). We also have an isomorphism induced by inclusion: 
\[
R_{F,b}^{\times} / \mathcal{O}_F^{\times} \cong (R_{F,b}^{(\mathfrak{p})})_+^{\times}/   (\mathcal{O}_F^{(\mathfrak{p})})^{\times} 
\]
and hence $e_b = e_b^{(\mathfrak{p})}$ for all $b \in B_F^{\times} \backslash \widehat{B}_F^{\times} / \widehat{R}_F^{\times}$.

\bigskip

Now as before $\Phi_E$ is scalar-valued automorphic eigenform (with trivial central character) with respect to $B_F^{\times}$ that corresponds to $\mathbf{f}_E$ under the Jacquet-Langlands correspondence, chosen so that it takes values in $\mathbf{Q}$. Let $t$ be a non-zero integer such that 
\[
t \cdot \frac{1}{e_b} \Phi_E(b) \in \mathbf{Z}
\]
for all $b \in B_F^{\times} \backslash \widehat{B}_F^{\times} / \widehat{R}_F^{\times}$. Then again:
\[
\zeta_{\Phi_E} := \sum_{b \in B_F^{\times} \backslash \widehat{B}_F^{\times} / \widehat{R}_F^{\times}}  ( t \cdot \frac{1}{e_b} \Phi_E(b)) \cdot [b]
\]
is a degree zero divisor on $B_F^{\times} \backslash \widehat{B}_F^{\times} / \widehat{R}_F^{\times}$. 

\bigskip

Now $E/F$ has conductor equal to $\mathfrak{n}$, and so $\mathbf{f}_E$ and hence $\Phi_E$ is new at $\mathfrak{p}$. Then $\Phi_E$ is in the kernel of the degeneracy maps from level $\mathfrak{n}$ to level $\mathfrak{n} \mathfrak{p}^{-1}$, {\it c.f.} for example section 1.5.7 of [N], and hence the same is true for $\zeta_{\Phi_E}$. Then $\zeta_{\Phi_E}$ is a non-zero element of $\mathfrak{X}_{\mathfrak{p}}(\mathfrak{n} \mathfrak{p}^{-1}, \mathfrak{p})$. One has $(\zeta_{\Phi_E},\zeta_{\Phi_E})_{\mathfrak{p}} = t^2 \langle \Phi_E,\Phi_E \rangle$. In addition since $\Phi_E$ corresponds to $\mathbf{f}_E$ under the Jacquet-Langlands correspondence, one has that $\zeta_{\Phi_E}$ belongs to the $\mathbf{f}_E$-eigen-submodule of $\mathfrak{X}_{\mathfrak{p}}(\mathfrak{n} \mathfrak{p}^{-1}, \mathfrak{p})$, and hence 
\[
\zeta_{\Phi_E} = s \cdot \xi_{\mathbf{f}_E}
\] 
for some non-zero integer $s$. In particular $(\zeta_{\Phi_E},\zeta_{\Phi_E})_{\mathfrak{p}}=s^2 (\xi_{\mathbf{f}_E},\xi_{\mathbf{f}_E})_{\mathfrak{p}}$. Consequently, we have
\[
(\xi_{\mathbf{f}_E},\xi_{\mathbf{f}_E})_{\mathfrak{p}} =  \langle \Phi_E,\Phi_E \rangle   \mod{(\mathbf{Q}^{\times})^{2}}.
\]

\bigskip

Thus to summarize, we have:
\begin{proposition}
In the case where $[F:\mathbf{Q}]$ is even, suppose that some prime ideal $\mathfrak{p}$ dividing $\mathfrak{n}_-$ is inert over $\mathbf{Q}$. Then we have, with notations as above:
\[
\deg(h) =  \big( \prod_{\mathfrak{q} | \mathfrak{n}_- } \overline{c}_{\mathfrak{q}}(E/F)  \big) \cdot \langle \Phi_E,\Phi_E \rangle   \mod{(\mathbf{Q}^{\times})^{2}}
\]
\end{proposition}

We can now complete the proof of Lemma 3.4. Namely that $E/\mathbf{Q}$ is an elliptic curve whose conductor is equal to $N = M \mathcal{Q}$, with $M,\mathcal{Q}$ relatively prime, and such that $\mathcal{Q}$ is a square-free product of an odd number of distinct prime numbers. We take $F/\mathbf{Q}$ to be real quadratic field as in the setting of section 3; thus all primes dividing $M$ splits in $F$, while all primes dividing $\mathcal{Q}$ are inert in $F$. Consider the base change $E/F$, which is an elliptic curve of conductor equal to $\mathfrak{n} = N \mathcal{O}_F$. Take $\mathfrak{n}_+ = M \mathcal{O}_F,\mathfrak{n}_- =\mathcal{Q} \mathcal{O}_F$. Then $\mathfrak{n} = \mathfrak{n}_+ \mathfrak{n}_-$ is an admissible factorization, and the prime ideals dividing $\mathfrak{n}_-$ lie over bijectively with the primes dividing $\mathcal{Q}$, and are inert over the corresponding prime numbers dividing $\mathcal{Q}$.

\bigskip
The elliptic curve $E/F$ is associated to the parallel weight two cuspidal Hilbert eigenform $\mathbf{f} =\mathbf{f}_E$ of level $\mathfrak{n}$, given by the base change of $f=f_E$ from $\GL_{2/\mathbf{Q}}$ to $\GL_{2/F}$. Proposition 4.6 is applicable in this setting, and we see that the proof of (3.19) is completed, simply on noting that, since $E/F$ is the base change of $E/\mathbf{Q}$, we have for each prime $q | \mathcal{Q}$:
\[
\overline{c}_q(E/\mathbf{Q}) = \overline{c}_{\mathfrak{q}} (E/F), \,\ \mbox{ where } \mathfrak{q} = q \mathcal{O}_F.
\]

\end{document}